\documentclass[12pt,reqno]{amsart}
\usepackage{amssymb}
\usepackage{mathrsfs}
\usepackage{palatino}
\usepackage{bbm}
\def\HC{{\operatorname{HC}}}

\newcommand{\CC}{\mathbb{C}}

\newcommand{\NN}{\mathbb{N}}

\newcommand{\fg}{\mathfrak{g}}

\newcommand{\kk}{\mathbbm{k}}

\newcommand{\gl}{\mathfrak{gl}}

\DeclareMathOperator{\ad}{ad}

\DeclareMathOperator{\qdet}{qdet}

\DeclareMathOperator{\Char}{char}

\DeclareMathOperator{\gr}{gr}

\DeclareMathOperator{\sgn}{sgn}
\DeclareMathOperator{\Mat}{Mat}

\usepackage{mathrsfs}
\usepackage{cases}
\usepackage{epic}
\usepackage{amsfonts}
\usepackage{graphicx}
\usepackage{amsmath}
\usepackage{amssymb, upgreek}
\usepackage{bm}
\usepackage{latexsym,todonotes}
\usepackage{pdflscape}
\usepackage[all]{xypic}
\usepackage[all]{xy}
\usepackage{color}
\usepackage{colordvi}
\usepackage{multicol}
\usepackage[normalem]{ulem}
\usepackage[linktocpage=true]{hyperref}
\usepackage{hyperref}
\hypersetup{colorlinks,linkcolor=blue,urlcolor=cyan,citecolor=red}
\usepackage{pictex}
\usepackage{amscd}
\usepackage{latexsym}
\usepackage{amssymb}
\usepackage{eucal}
\usepackage{eufrak}
\usepackage{verbatim}

\numberwithin{equation}{section}
\newtheorem{Theorem}{Theorem}[section]
\newtheorem{Lemma}[Theorem]{Lemma}
\newtheorem{Corollary}[Theorem]{Corollary}
\newtheorem{Proposition}[Theorem]{Proposition}
\newtheorem{Remark}[Theorem]{Remark}

\theoremstyle{Theorem}

\newtheorem*{thm*}{Theorem}
\newtheorem*{thm**}{Corollary}
\newtheorem*{thm***}{Theorem B}
\theoremstyle{remark}

\numberwithin{equation}{section}

\begin{document}

\title[Parabolic generators]{The center of modular shifted Yangians and parabolic generators}
\author[Hao Chang \lowercase{and} Hongmei Hu]{Hao Chang \lowercase{and} Hongmei Hu*}
\address[Hao Chang]{School of Mathematics and Statistics, 
Hubei Key Laboratory of Mathematical Sciences, 
and Key Laboratory of Nonlinear Analysis \& Applications (Ministry of Education), Central China Normal University, Wuhan 430079, China}
\email{chang@ccnu.edu.cn}
\address[Hongmei Hu]{Department of Mathematics, Shanghai Maritime University, Shanghai 201306, China}
\email{hmhu@shmtu.edu.cn}
\thanks{* Corresponding author.}

\makeatletter
\makeatother


\begin{abstract} 
This paper is devoted to the study of the shifted Yangian $Y_n(\sigma)$ associated to the general linear Lie algebra $\gl_n$ over a field of positive characteristic.
We obtain an explicit description of the center $Z(Y_n(\sigma))$ of $Y_n(\sigma)$ in terms of parabolic generators,
showing that it is generated by its Harish-Chandra center and its $p$-center.
\end{abstract}

\maketitle
\section{Introduction}
The Yangians were introduced by Drinfeld in his fundamental paper \cite{D85}.
They form a remarkable family of quantum groups related to rational solutions of the classical {\em Yang-Baxter} equation.
The Yangian $Y_n$ associated to the Lie algebra $\gl_n$ was earlier considered in the work
of mathematical physicists from St. Petersburg \cite{TF79}. 
It is an associative algebra whose defining relations can be written in a specific matrix form, which is called the {\em RTT relation}; see e.g. \cite{MNO96}.
For related topics and further applications of Yangians,
the reader is referred to \cite{Mol07} and the references cited therein.

In \cite{BK05},
Brundan and Kleshchev found a {\em parabolic presentation} for $Y_n$ associated to each composition $\mu$ of $n$. 
The new presentation corresponds to a block matrix decomposition of $\gl_n$ of shape $\mu$. 
In the special case when $\mu=(n)$, 
this presentation is exactly the original {\em RTT presentation}. 
On the other extreme case when $\mu=(1,\dots,1)=(1^n)$, 
the corresponding parabolic presentation is just a variation of Drinfeld’s from \cite{D88};
see \cite[Remark 5.12]{BK05}. 
The parabolic presentations have played a crucial role in their subsequent works.
Over the field of complex numbers, 
Brundan–Kleshchev \cite{BK06} introduced {\em shifted Yangians} and their truncated analogues,
which are isomorphic to the {\em finite $W$-algebras} associated to nilpotent orbits in $\gl_n$,
as defined by Premet \cite{Pre02}.
This allowed them to
make an extensive study of the representation theory of these finite $W$-algebras in
\cite{BK08}.

In characteristic zero, it is well-known that the center $Z(Y_n)$ of $Y_n$ is generated by
the coefficients of the {\em quantum determinant} $\qdet T(u)$, see \cite[Theorem 2.13]{MNO96}.
The quantum determinant $\qdet T(u)$ admits a factorization in terms of the diagonal Drinfeld generators (\cite[Theorem 8.6]{BK05},
see also \cite[Theorem 1.10.5]{Mol07}).
Moreover, the authors proved that $\qdet T(u)$ can also be decomposed into the product
in terms of the diagonal parabolic generators \cite[Proposition 3.2]{CH23-1}.

In \cite{BT18}, Brundan and Topley developed the theory of shifted Yangians $Y_n(\sigma)$ over a field of positive characteristic. 
In particular, they gave a modular version of a Drinfeld-type presentation of $Y_n(\sigma)$. 
One of the key features which differs from characteristic zero is the existence of a large central subalgebra $Z_p(Y_n(\sigma))$, 
called the {\em $p$-center}.
These results give important applications to the theory of {\em modular finite
$W$-algebras} \cite{GT2019}.
The authors observed that Brundan–Kleshchev’s isomorphism can be reduced to positive characteristic,
i.e., the modular finite $W$-algebra is a truncation of the modular shifted Yangian \cite{GT19}. 
Both algebras admit a $p$-center,
and the restricted version of Brundan–Kleshchev’s isomorphism was established by Goodwin and
Topley (see \cite{GT21}).

The main goal of this article is to obtain the generalization of the $p$-center of \cite{BT18}
for the modular shifted Yangian $Y_n(\sigma)$ in terms of parabolic generators.
Our approach is basically generalizing the arguments in \cite{BT18} with suitable modifications,
however there are other technical obstacles that did not appear in the Drinfeld presentation and
we need to do more complicated computations when treating the parabolic presentation.
We expect the results of this article will be used on more connections 
between the modular shifted Yangian $Y_n(\sigma)$ and the modular finite $W$-algebras
in future work (cf. \cite[Theorem 10.1]{BK06}, \cite[Theorem 4.3]{GT19}).
It would also be interesting to generalize the results in this article to the shifted super Yangian $Y_{m|n}(\sigma)$ (see \cite{Peng21}, \cite{CH23-2}).

We give a quick outline of our main results for the Yangian $Y_n$,
that is, the case $\sigma=0$.
The Yangian $Y_n$ over an algebraically closed field $\kk$ of positive characteristic
is defined to be the associative algebra with the usual RTT presentation.
It is a filtered deformation of the universal enveloping algebra $U(\gl_n[t])$ of the polynomial current Lie algebra
$\gl_n[t]:=\gl_n\otimes \kk[t]$.

In positive characteristic, 
the coefficients of the quantum determinant $\qdet T(u)$ still belong to the center $Z(Y_n)$.
The algebra generated by the coefficients will be denoted by $Z_\HC(Y_n)$ which is a subalgebra of $Z(Y_n)$.
We call it the {\em Harish-Chandra center} of $Y_n$.
Suppose that $\Char{\kk}=:p>0$.
The current algebra $\gl_n[t]$ admits a natural {\em restricted Lie algebra} structure,
with $p$-map $x\mapsto x^{[p]}$.
Given a composition $\mu=(\mu_1,\dots,\mu_m)$ of $n$ of length $m$,
we decompose the $\gl_n[t]$ into block matrices according to the composition $\mu$.
There results a basis of $\gl_n[t]$ (see \eqref{shifted current algebra spanned set}):
\[
\{e_{a,b;i,j}t^r;~1\leq a, b\leq m,1\leq i\leq\mu_a, 1\leq j\leq \mu_b, r\geq 0\}.
\]
Then $p$-map is defined on the basis by the rule $(e_{a,b;i,j}t^r)^{[p]}:=\delta_{a,b}\delta_{i,j}e_{a,b;i,j}t^{rp}$.
By general theory (cf. \cite[Section 2.3]{Jan97}),
the enveloping algebra $U(\gl_n[t])$ has a large $p$-center $Z_p(\gl_n[t])$ generated by the elements
\begin{align}
\{(e_{a,b;i,j}t^r)^p-\delta_{a,b}\delta_{i,j}e_{a,b;i,j}t^{rp};~1\leq a, b\leq m,1\leq i\leq\mu_a, 1\leq j\leq \mu_b, r\geq 0\}.  
\end{align}
It is natural to look for lifts of the $p$-central elements in $Z(Y_n)$.
In Section \ref{section name:Centers},
we will give a description of the $p$-center $Z_p(Y_n)$ in terms of parabolic generators 
and show that the generators of $Z_p(Y_n)$ provide the lifts of generators for $Z_p(\gl_n[t])$.
We show in particular that the center $Z(Y_n)$ is generated by $Z_\HC(Y_n)$ and $Z_p(Y_n)$.

We organize this article in the following manner. 
In Section \ref{section name: shifted current}, 
we recall some notations and properties about the shifted current Lie algebra in the parabolic setting.
We introduce the modular Yangians
and shifted Yangians in Section \ref{section name:modular Yangians},
and extend the parabolic presentation from characteristic zero to positive characteristic. 
Section \ref{section name:Centers} is concerned with the center of $Y_n(\sigma)$. 
We investigate various $p$-central elements by employing the parabolic presentation. 
Moreover, we give a description
of the center of $Y_n(\sigma)$ and obtain the precise formulas for the generators. 
We also show that the $p$-center of $Y_n(\sigma)$ is independent of the choice of the shape $\mu$.

\emph{Throughout this paper, $\kk$ denotes an algebraically closed field of characteristic $\Char(\kk)=:p>2$.}

\section{The shifted current algebra}\label{section name: shifted current}
In this section, we recall the definition {\em shifted current algebra} and describe the center of its enveloping algebra.
Most results follow from \cite[Section 3]{BT18}, 
we slightly change the notation in the {\em parabolic} setting for our later studies. 

\subsection{Shift matrix and shifted current algebra}\label{Section: shift matrix and shifted currentalgebra}
A {\em shift matrix} is an $n\times n$ array $\sigma=(s_{i,j})_{1\leq i,j\leq n}$ of non-negative integers satisfying
\begin{align}\label{shift matrix condition}
s_{i,j}+s_{j,k}=s_{i,k}    
\end{align}
whenever $|i-j|+|j-k|=|i-k|$.

Let $\fg:=\gl_n[t]=\gl_n\otimes \kk[t]$ be the {\em current Lie algebra},
$U(\fg)$ be its universal enveloping algebra.
The elements $e_{i,j}t^r:=e_{i,j}\otimes t^r$ with $r\geq 0$ and $i,j=1,\dots, n$ make a basis of $\fg$,
and the Lie bracket with $r,s\geq 0$ is given by
\begin{align}\label{Lie bracket of current Lie}
 [e_{i,j}t^r,e_{k,l}t^s]=\delta_{k,j}e_{i,l}t^{r+s}-\delta_{l,i}e_{k,j}t^{r+s}.
\end{align}

Let $\mu=(\mu_1,\dots,\mu_m)$ be a given composition of $n$ of length $m$.
For each $1\leq a\leq m$, we define the partial sum
\begin{align}\label{notation:partial sum}
p_a(\mu):=\mu_1+\cdots+\mu_{a-1}.
\end{align}
Assume that $\sigma=(s_{i,j})_{1\leq i,j\leq n}$ is a fixed shift matrix of size $n$.
We say a composition $\mu=(\mu_1,\dots,\mu_m)$ is {\em admissible} to $\sigma$ if $s_{i,j}=0$
for all $p_a(\mu)+1\leq i,j\leq p_{a+1}(\mu)$ and $1\leq a\leq m$.
When $\mu=(\mu_1,\dots,\mu_m)$ is admissible to $\sigma$, we will use a shorthand notation
\begin{align}\label{sab mu}
s_{a,b}^{\mu}:=s_{p_{a+1}(\mu),p_{b+1}(\mu)}=s_{\mu_1+\cdots+\mu_{a},\mu_1+\cdots+\mu_b},~\ \ \forall~1\leq a,b\leq m.
\end{align}
Note that one can recover the original matrix $\sigma$ if an admissible shape $\mu$ and the numbers $\{s_{a,b}^{\mu};~1\leq a,b\leq m\}$ are known.
The admissible condition in conjunction with \eqref{shift matrix condition} implies that for any $1\leq a, b\leq m$, we have
\begin{align}\label{sab mu=si,j}
s_{p_a(\mu)+i,p_b(\mu)+j}=s_{a,b}^{\mu},~\ \ \forall~1\leq i \leq \mu_a, 1\leq j\leq \mu_b.
\end{align}

Given a composition $\mu=(\mu_1,\dots,\mu_m)$, we define 
\[
e_{a,b;i,j}t^r:=e_{p_a(\mu)+i,p_b(\mu)+j}t^r=e_{\mu_1+\cdots+\mu_{a-1}+i,\mu_1+\cdots+\mu_{b-1}+j}t^r\in\fg.
\]
Using \eqref{Lie bracket of current Lie},
we compute immediately
\begin{align}\label{Lie bracket current lie 2}
 [e_{a,b;i,j}t^r,e_{c,d;k,l}t^s]=\delta_{b,c}\delta_{k,j}e_{a,d;i,l}t^{r+s}-\delta_{d,a}\delta_{l,i}e_{c,b;k,j}t^{r+s}.
\end{align}

Let $\sigma$ be a fixed shift matrix, $\mu=(\mu_1,\dots,\mu_m)$ an admissible shape to $\sigma$.
The {\em shifted current algebra} is $\fg_{\sigma}\subseteq\fg$ spanned by
\begin{align}\label{shifted current algebra spanned set}
\{e_{a,b;i,j}t^r;~1\leq a,b\leq m, 1\leq i\leq \mu_a, 1\leq j\leq\mu_b, r\geq s_{a,b}^{\mu}\}.
\end{align}
Since $\mu$ is admissible, 
\eqref{sab mu=si,j} implies that the definition of $\fg_\sigma$ coincides with the definition given in \cite[Section 3.2]{BT18}.
Hence it is independent of the particular choice of the admissible shape $\mu$.
Moreover,
the shifted current algebra $\fg_\sigma$ is a restricted Lie subalgebra of $\fg$, see \cite[Section 3.4]{BT18} for more details.

\subsection{The center of $U(\fg_{\sigma})$}
There is one obvious family of central elements in the universal enveloping algebra $U(\fg_\sigma)$.
For any $r\in\NN$, we set
\[
z_r:=\sum\limits_{i=1}^{\mu_1}e_{1,1;i,i}t^r+\cdots+\sum\limits_{i=1}^{\mu_m}e_{m,m;i,i}t^r\in\fg_{\sigma}.
\]
Denote by $Z(\fg_\sigma):=Z(U(\fg_{\sigma}))$ the center of $U(\fg_\sigma)$. 
It follows that $\kk[z_r;~r\geq 0]$ is a subalgebra of $Z(\fg_\sigma)$.
As $\fg_\sigma$ is restricted, the {\em $p$-center} $Z_p(\fg_\sigma)$ of $U(\fg_\sigma)$ can be generated by the set
\begin{align}\label{generator of p-centerof ug sigma}
\{(e_{a,b;i,j}t^r)^p-\delta_{a,b}\delta_{i,j}e_{a,b,;i,j}t^{rp};~1\leq a,b\leq m, 1\leq i\leq \mu_a, 1\leq j\leq\mu_b, r\geq s_{a,b}^{\mu}\}.
\end{align}
\begin{Theorem}\cite[Theorem 3.4]{BT18}\label{theorem:center of ugsigma}
The center $Z(\fg_\sigma)$ of $U(\fg_\sigma)$ is generated by $\{z_r;~r\geq 0\}$ and $Z_p(\fg_\sigma)$.
In fact, $Z(\fg_\sigma)$ is freely generated by 
\begin{align*}
\{z_r;~r\geq 0\}\cup\{(e_{a,b;i,j}t^r)^p-\delta_{a,b}\delta_{i,j}e_{a,b,;i,j}t^{rp};~1\leq a,b\leq m, & 1\leq i\leq \mu_a, 1\leq j\leq\mu_b\\
&\ \ \ \  ~\text{with}~(a,b,i,j)\neq (1,1,1,1), r\geq s_{a,b}^{\mu}\}.
\end{align*}
\end{Theorem}

\section{Modular Yangians}\label{section name:modular Yangians}
\subsection{The RTT generators}\label{subsection:rtt generators}
The Yangian associated to the general linear Lie algebra $\gl_n$,
denoted by $Y_n$, is the associative algebra over $\kk$ with the {\em RTT generators}
$\{t_{i,j}^{(r)};~1\leq i,j\leq n,~r>0\}$ subject the following relations:
\begin{align}\label{RTT relations}
\left[t_{i,j}^{(r)}, t_{k,l}^{(s)}\right] =\sum_{t=0}^{\min(r,s)-1}
\left(t_{k, j}^{(t)} t_{i,l}^{(r+s-1-t)}-
t_{k,j}^{(r+s-1-t)}t_{i,l}^{(t)}\right)
\end{align}
for $1\leq i,j,k,l\leq n$ and $r,s>0$.
By convention, we set $t_{i,j}^{(0)}:=\delta_{i,j}$.
We often put the generators $t_{i,j}^{(r)}$ for all $r\geq 0$ together to form the generating function
\[
t_{i,j}(u):= \sum_{r \geq 0}t_{i,j}^{(r)}u^{-r} \in Y_n[[u^{-1}]].
\]
Then these power series for all $1\leq i,j\leq n$ can be collected together into a single matrix $T(u):=(t_{i,j}(u))_{1\leq i,j\leq n}\in\Mat_n(Y_n[[u^{-1}]])$.
It is easily seen that, in terms of the generating series,
the initial defining relation \eqref{RTT relations} may be rewritten as follows:
\begin{align}\label{tiju tklv relation}
(u-v)[t_{i,j}(u),t_{k,l}(v)]=t_{k,j}(u)t_{i,l}(v)-t_{k,j}(v)t_{i,l}(v).
\end{align}

We record the following fundamental {\em PBW theorem} (\cite[Theorem 4.1]{BT18}).
\begin{Theorem}\label{Theorem: PBW-RTT}
Ordered monomials in the generators $\{t_{i,j}^{(r)};~1\leq i,j\leq n, r>0\}$
taken in any fixed order given a basis of $Y_n$.
\end{Theorem}

Define the {\em loop filtration}
\begin{align}\label{loop filtration}
Y_n=\bigcup\limits_{r\geq 0}{\rm F}_r Y_n
\end{align}
by setting $\deg t_{i,j}^{(r)}=r-1$. Denote by $\gr Y_n$ the associated graded algebra of $Y_n$.

The following is a consequence of the PBW theorem (cf. \cite[Lemma 4.2]{BT18}).
\begin{Lemma}\label{Lemma:chi iso}
There exists an isomorphism $\chi:U(\fg)\rightarrow\gr Y_n$ sending $e_{i,j}t^r$ to $\gr_r t_{i,j}^{(r+1)}$ for each $1\leq i,j\leq n$ and $r\geq 0$.
\end{Lemma}

\subsection{Gauss decomposition and quasideterminants}\label{subsection name:Gauss decomp}
Let $\mu=(\mu_1,\dots,\mu_m)$ be a composition of $n$. 
Since the leading minors of the $n\times n$ matrix $T(u)$ are invertible,
it possesses a {\em Gauss factorization}
\begin{align}\label{gauss decomp}
  T(u)=F(u)D(u)E(u)  
\end{align}
for unique {\em block matrices} $D(u)$, $E(u)$ and $F(u)$ of the form
$$
D(u) = \left(
\begin{array}{cccc}
D_{1}(u) & 0&\cdots&0\\
0 & D_{2}(u) &\cdots&0\\
\vdots&\vdots&\ddots&\vdots\\
0&0 &\cdots&D_{m}(u)
\end{array}
\right),
$$

$$
E(u) =
\left(
\begin{array}{cccc}
I_{\mu_1} & E_{1,2}(u) &\cdots&E_{1,m}(u)\\
0 & I_{\mu_2} &\cdots&E_{2,m}(u)\\
\vdots&\vdots&\ddots&\vdots\\
0&0 &\cdots&I_{\mu_{m}}
\end{array}
\right),
F(u) = \left(
\begin{array}{cccc}
I_{\mu_1} & 0 &\cdots&0\\
F_{2,1}(u) & I_{\mu_2} &\cdots&0\\
\vdots&\vdots&\ddots&\vdots\\
F_{m,1}(u)&F_{m,2}(u) &\cdots&I_{\mu_{m}}
\end{array}
\right),
$$
where
\[
D_a(u)=\big(D_{a;i,j}(u)\big)_{1 \leq i,j \leq \mu_a},
\]
\[
E_{a,b}(u)=\big(E_{a,b;i,j}(u)\big)_{1 \leq i \leq \mu_a, 1 \leq j \leq \mu_b},
\]
\[
F_{b,a}(u)=\big(F_{b,a;i,j}(u)\big)_{1 \leq i \leq \mu_b, 1 \leq j \leq \mu_a}
\]
are $\mu_a \times \mu_a$,
$\mu_a \times \mu_b$
and $\mu_b \times\mu_a$ matrices, respectively. 
Also define the $\mu_a\times\mu_a$ matrix
$D_a^{\prime}(u)=\big(D_{a;i,j}^{\prime}(u)\big):=(D_a(u)\big)^{-1}$
This defines power series
\[
D_{a;i,j}(u) =\sum_{r \geq 0} D_{a;i,j}^{(r)}u^{-r},~
D'_{a;i,j}(u) =\sum_{r \geq 0} D'^{(r)}_{a;i,j}u^{-r}
\]
and
\[
E_{a,b;i,j}(u)=\sum_{r \geq 1} E_{a,b;i,j}^{(r)} u^{-r},~
F_{b,a;i,j}(u) =\sum_{r \geq 1} F_{b,a;i,j}^{(r)} u^{-r}.
\]
Let 
\[
E_{a;i,j}(u)=\sum_{r \geq 1} E_{a;i,j}^{(r)}u^{-r}:=E_{a,a+1;i,j}(u),~
F_{a;i,j}(u) =\sum_{r \geq 1} F_{a;i,j}^{(r)}u^{-r}:=F_{a+1,a;i,j}(u)
\]
for short.
When necessary, we will
add an additional superscript $\mu$ to our notation to avoid any ambiguity as $\mu$ varies.
In terms of quasideterminants of \cite{GR97},
we have the following more explicit descriptions, as noted already in \cite[Section 6]{BK05}.
We write the matrix $T(u)$ in block form as
\[
T(u) = \left(
\begin{array}{lll}
{^\mu}T_{1,1}(u)&\cdots&{^\mu}T_{1,m}(u)\\
\vdots&\ddots&\cdots\\
{^\mu}T_{m,1}(u)&\cdots&{^\mu}T_{m,m}(u)\\
\end{array}
\right),
\]
 where each ${^\mu}T_{a,b}(u)$ is a $\mu_a \times \mu_b$ matrix.
Then
\begin{align}\label{quasiD}
D_a(u) =
\left|
\begin{array}{cccc}
{^\mu}T_{1,1}(u) & \cdots & {^\mu}T_{1,a-1}(u)&{^\mu}T_{1,a}(u)\\
\vdots & \ddots &\vdots&\vdots\\
{^\mu}T_{a-1,1}(u)&\cdots&{^\mu}T_{a-1,a-1}(u)&{^\mu}T_{a-1,a}(u)\\
{^\mu}T_{a,1}(u) & \cdots & {^\mu}T_{a,a-1}(u)&
\hbox{\begin{tabular}{|c|}\hline${^\mu}T_{a,a}(u)$\\\hline\end{tabular}}
\end{array}
\right|,
\end{align}
\begin{align}\label{quasiE}
E_{a,b}(u) =
D^{\prime}_a(u)
\left|\begin{array}{cccc}
{^\mu}T_{1,1}(u) & \cdots &{^\mu}T_{1,a-1}(u)& {^\mu}T_{1,b}(u)\\
\vdots & \ddots &\vdots&\vdots\\
{^\mu}T_{a-1,1}(u) & \cdots & {^\mu}T_{a-1,a-1}(u)&{^\mu}T_{a-1,b}(u)\\
{^\mu}T_{a,1}(u) & \cdots & {^\mu}T_{a,a-1}(u)&
\hbox{\begin{tabular}{|c|}\hline${^\mu}T_{a,b}(u)$\\\hline\end{tabular}}
\end{array}
\right|,
\end{align}
\begin{align}\label{quasiF}
F_{b,a}(u) =
\left|
\begin{array}{cccc}
{^\mu}T_{1,1}(u) & \cdots &{^\mu}T_{1,a-1}(u)& {^\mu}T_{1,a}(u)\\
\vdots & \ddots &\vdots&\vdots\\
{^\mu}T_{a-1,1}(u) & \cdots & {^\mu}T_{a-1,a-1}(u)&{^\mu}T_{a-1,a}(u)\\
{^\mu}T_{b,1}(u) & \cdots & {^\mu}T_{b,a-1}(u)&
\hbox{\begin{tabular}{|c|}\hline${^\mu}T_{b,a}(u)$\\\hline\end{tabular}}
\end{array}
\right|D^{\prime}_a(u).
\end{align}
By the above, we immediately have 
$E_{b-1;i,j}^{(1)}=T_{b-1,b;i,j}^{(1)}$ and $F_{b-1;i,j}^{(1)}=T_{b,b-1;i,j}^{(1)}$ for all admissible $b, i, j$,
By induction, one may show that for each pair $a,b$ such that $1<a+1<b\leq m-1$ and $1\leq i\leq \mu_a$, $1\leq j\leq \mu_b$, we have
\begin{align}\label{EFabij r induction formula}
E_{a,b;i,j}^{(r)}=[E_{a,b-1;i,k}^{(r)},E_{b-1;k,j}^{(1)}],~F_{b,a;j,i}^{(r)}=[F_{b-1;j,k}^{(1)},F_{b-1,a;k,i}^{(r)}],
\end{align}
for any $1\leq k\leq\mu_{b-1}$.

By multiplying out the matrix $T(u)=F(u)D(u)E(u)$,
we see that each $t_{i,j}^{(r)}$ can be expressed as a sum of monomials in $D_{a;i,j}^{(r)}$,
$E_{a,b;i,j}^{(r)}$ and $F_{b,a;i,j}^{(r)}$ in a certain order with all $F$'s before $D$'s and all $D$'s before $E$'s.
By \eqref{EFabij r induction formula},
it is enough to use $D_{a;i,j}^{(r)}$,
$E_{a;i,j}^{(r)}$ and $F_{a;i,j}^{(r)}$ only. Consequently, we have the following theorem
\begin{Theorem}\label{Theorem: yn gb parabolic generators}
The Yangian $Y_n$ is generated as an algebra by the elements
$\{D_{a;i,j}^{(r)};~1 \leq a \leq m,  1 \leq i,j \leq \mu_a,r \geq 0\}$,
$\{E_{a;i,j}^{(r)};~1 \leq a < m, 1 \leq i \leq \mu_a,1 \leq j \leq \mu_{a+1}, r \geq 1\}$ and
$\{F_{a;i,j}^{(r)};~1 \leq a < m, 1 \leq i \leq \mu_{a+1},
1 \leq j \leq \mu_a, r \geq 1\}$.
\end{Theorem}

The above generators are called {\em parabolic generators}.
In the special case when $\mu=(n)$, 
we have $T(u) ={^\mu}T_{1,1}(u)$. 
The corresponding parabolic generators are exactly the original RTT generators in subsection \ref{subsection:rtt generators}. 
At another extreme, for $\mu =(1,\dots,1)$, 
the parabolic generators are called the {\em Drinfeld generators}. 
In this case, all $D_a(u)$’s, $E_{a,b}(u)$'s and $F_{b,a}(u)$'s are $1\times1$ matrices. 
For conciseness,
we will use $d_a(u)$, $e_{a,b}(u)$, $e_a(u)$, $f_{b,a}(u)$ and $f_a(u)$ instead of $D_a(u)$, $E_{a,b}(u)$, $E_{a,a+1}(u)$, $F_{b,a}(u)$ and $F_{a+1,a}(u)$, respectively.
\subsection{Shift map}
Using the RTT relation (cf. \cite[(2.4)]{BK05}), one can show that the map $\omega_n:Y_n\rightarrow Y_n;~T(u)\mapsto T(-u)^{-1}$ is an algebra automorphism (see \cite[Proposition 1.12(iii)]{MNO96}).
For any $k\geq 0$, we let $\varphi_k:Y_n\hookrightarrow Y_{k+n}$ denote the obvious injective algebra homomorphism 
mapping $t_{i,j}^{(r)}\in Y_n$ to $t_{k+i,k+j}^{(r)}\in Y_{k+n}$.
Then the map $\psi_k:Y_n\rightarrow Y_{k+n}$ defined by $\psi_k:=\omega_{k+n}\circ\varphi_k\circ\omega_n$ is an injective algebra homomorphism (cf. \cite[Section 4]{BK05}). 
We call $\psi_k$ the {\em shift map}.

We record the following Lemma, 
where the proof given in characteristic zero in \cite{BK05} works as well in positive characteristic.
\begin{Lemma}\label{Lemma:property of psi}
Let $\psi$ be the shift map. Then the following statements hold:
\begin{enumerate}
\item[(1)] For $k,l\geq 1$, we have that
\begin{align}\label{psik d}
d_{k+l}(u)=\psi_{k}(d_l(u)),
\end{align}
\begin{align}\label{psik e}
 e_{k+l}(u)=\psi_{k}(e_l(u)),
 \end{align}
\begin{align}\label{psik f}
f_{k+l}(u)=\psi_{k}(f_l(u))
 \end{align}
\item[(2)] Fix $a\geq 1$ and let $\bar{\mu}:=(\mu_a,\mu_{a+1},\dots,\mu_m)$.
Then for all admissible $i,j$, we have that 
\begin{align}\label{psik D}
{^\mu}D_{a;i,j}(u)=\psi_{\mu_1+\cdots+\mu_{a-1}}({^{\bar{\mu}}}D_{1;i,j}(u)).
 \end{align}
 \begin{align}\label{psik E}
{^\mu}E_{a;i,j}(u)=\psi_{\mu_1+\cdots+\mu_{a-1}}({^{\bar{\mu}}}E_{1;i,j}(u)).
 \end{align}
 \begin{align}\label{psik F}
{^\mu}F_{a;i,j}(u)=\psi_{\mu_1+\cdots+\mu_{a-1}}({^{\bar{\mu}}}F_{1;i,j}(u)).
 \end{align}
\end{enumerate}
\end{Lemma}
\subsection{Parabolic presentation}
The following theorem is the modular statement of \cite[Theorem A]{BK05}.
\begin{Theorem}\label{Theorem:parabolic presentation}
The algebra $Y_n$ is generated by elements
$\{D_{a;i,j}^{(r)},{D}'^{(r)}_{a;i,j};~1 \leq a \leq m,  1 \leq i,j \leq \mu_a,r \geq 0\}$,
$\{E_{a;i,j}^{(r)};~1 \leq a < m, 1 \leq i \leq \mu_a,1 \leq j \leq \mu_{a+1}, r \geq 1\}$ and
$\{F_{a;i,j}^{(r)};~1 \leq a < m, 1 \leq i \leq \mu_{a+1},
1 \leq j \leq \mu_a, r \geq 1\}$ subject only to the following relations:
\begin{align}\label{pr1}
D_{a;i,j}^{(0)}= \delta_{i,j},
\end{align}
\begin{align}\label{pr2}
\sum_{t=0}^r D_{a;i,\alpha}^{(t)}D'^{(r-t)}_{a;\alpha,j}=\delta_{r,0}\delta_{i,j},
\end{align}
\begin{align}\label{pr3}
[D_{a;i,j}^{(r)}, D_{b;k,l}^{(s)}] 
=\delta_{a,b}
\sum_{t=0}^{\min(r,s)-1}
\left(
D_{a;i,l}^{(r+s-1-t)}D_{a;k,j}^{(t)}
-D_{a;i,l}^{(t)}D_{a;k,j}^{(r+s-1-t)} \right),
\end{align}
\begin{align}\label{pr4}
[E_{a;i,j}^{(r)},F_{b;k,l}^{(s)}]=-\delta_{a,b} \sum_{t=0}^{r+s-1}
D'^{(t)}_{a;i,l}D_{a+1;k,j}^{(r+s-1-t)},
\end{align}
\begin{align}\label{pr5}
[D_{a;i,j}^{(r)}, E_{b;k,l}^{(s)}] =\delta_{a,b}\sum_{t=0}^{r-1}D_{a;i,\alpha}^{(t)} E_{a;\alpha,l}^{(r+s-1-t)}\delta_{k,j}
-\delta_{a,b+1}\sum_{t=0}^{r-1}D_{b+1;i,l}^{(t)} E_{b;k,j}^{(r+s-1-t)},
\end{align}
\begin{align}\label{pr6}
[D_{a;i,j}^{(r)}, F_{b;k,l}^{(s)}]=
\delta_{a,b+1} \sum_{t=0}^{r-1}
 F_{b;i,l}^{(r+s-1-t)}D_{b+1;k,j}^{(t)}-\delta_{a,b}
\delta_{i,l}\sum_{t=0}^{r-1}
F_{a;k,\alpha}^{(r+s-1-t)}D_{a;\alpha,j}^{(t)},
\end{align}
\begin{align}\label{pr7}
[E_{a;i,j}^{(r)}, E_{a;k,l}^{(s)}]
=\sum_{t=1}^{s-1} E_{a;i,l}^{(t)} E_{a;k,j}^{(r+s-1-t)}
-\sum_{t=1}^{r-1} E_{a;i,l}^{(t)} E_{a;k,j}^{(r+s-1-t)},
\end{align}
\begin{align}\label{pr8}
[F_{a;i,j}^{(r)}, F_{a;k,l}^{(s)}]
=\sum_{t=1}^{r-1} F_{a;i,l}^{(r+s-1-t)}F_{a;k,j}^{(t)}-
\sum_{t=1}^{s-1} F_{a;i,l}^{(r+s-1-t)}F_{a;k,j}^{(t)},
\end{align}
\begin{align}\label{pr9}
[E_{a;i,j}^{(r+1)}, E_{a+1;k,l}^{(s)}]
-[E_{a;i,j}^{(s)}, E_{a+1;k,l}^{(r+1)}]
=\delta_{k,j}E_{a;i,\beta}^{(r)} E_{a+1;\beta,l}^{(s)},
\end{align}
\begin{align}\label{pr10}
[F_{a;i,j}^{(r)}, F_{a+1;k,l}^{(s+1)}]
-[F_{a;i,j}^{(r+1)}, F_{a+1;k,l}^{(s)}]
=\delta_{i,l}F_{a+1;k,\beta}^{(s)}F_{a;\beta,j}^{(r)},
\end{align}
\begin{align}\label{pr11}
[E_{a;i,j}^{(r)}, E_{b;k,l}^{(s)}] = 0
\quad\text{ if $b>a+1$ or if $b=a+1$ and $k\neq j$},
\end{align}
\begin{align}\label{pr12}
[F_{a;i,j}^{(r)}, F_{b;k,l}^{(s)}] = 0
\quad\text{ if $b>a+1$ or if $b=a+1$ and $i\neq l$},
\end{align}
\begin{align}\label{pr13}
[E_{a;i,j}^{(r)}, [E_{a;k,l}^{(s)}, E_{b;f,g}^{(t)}]]
+
[E_{a;i,j}^{(s)}, [E_{a;k,l}^{(r)}, E_{b;f,g}^{(t)}]] = 0
\quad\:\:\text{ if }|a-b|=1,
\end{align}
\begin{align}\label{pr14}
[F_{a;i,j}^{(r)}, [F_{a;k,l}^{(s)}, F_{b;f,g}^{(t)}]]
+
[F_{a;i,j}^{(s)}, [F_{a;k,l}^{(r)}, F_{b;f,g}^{(t)}]] = 0
\:\:\quad\text{ if }|a-b|=1,
\end{align}
for all admissible $a,b,f,g,i,j,k,l,r,s,t$.
(By convention the index $\alpha$ resp. $\beta$ appearing here should be summed over
$1,\dots,\mu_a$ resp. $1,\dots,\mu_{a+1}$.)
\vspace{2mm}   
\end{Theorem}
\begin{proof}
We first need to check that the relations \eqref{pr1}-\eqref{pr14} do indeed hold in $Y_n$.
These were proved in \cite{BK05} over the ground field $\CC$.
These calculations can be performed without any difference in positive characteristic,
see also the proof of \cite[Theorem 4.3]{BT18}.
The arguments start from the low rank cases, 
then transfer the corresponding relations to the general case of $Y_n$ by using the shift map and the transposition map (see Subsection \ref{section: auto} below).
For example, to establish \eqref{pr13},
we use the power series relation
\begin{align}\label{cubic serre 2ge relation}
[E_{a;i,j}(u),[E_{a;k,l}(v),E_{b;g,h}(w)]]+[E_{a;i,j}(v),[E_{a;k,l}(u),E_{b;g,h}(w)]]=0
\end{align}
for $|a-b|=1$.
Using the shift map, 
it is enough to prove \eqref{cubic serre 2ge relation} for the case $\{a,b\}=\{1,2\}$,
then argue as in \cite[Lemma 6.6]{BK05}.

Now let $\widehat{Y}_n$ be the algebra with generators and relations as in the statement of the theorem.
We may further define higher root elements $E_{a,b;i,j}^{(r)}$, $F_{b,a;j,i}^{(r)}\in\widehat{Y}_n$ by equations \eqref{EFabij r induction formula}.
Let $\theta:\widehat{Y}_n\rightarrow Y_n$ be the map sending every element in $\widehat{Y}_n$ into the element in $Y_n$ with the same name.
The previous paragraph implies that $\theta$ is a well-defined surjective homomorphism.
Therefore, it remains to prove that $\theta$ is also injective.

Let $\widehat{Y}_n^+$ denote the subalgebra of $\widehat{Y}_n$ generated by the elements $\{E_{a,b;i,j}^{(r)}\}$.
Define a filtration on $\widehat{Y}_n^+$ by declaring that the elements $E_{a,b;i,j}^{(r)}$ are of filtered degree $r-1$,
and denote by $\gr \widehat{Y}_n^+$ the corresponding graded algebra.
Let $e_{a,b;i,j}^{(r)}:=\gr_{r}E_{a,b;i,j}^{(r+1)}$.
One can show that the following identity holds in $\gr \widehat{Y}_n^+$.
\begin{align}\label{gryn+relation}
[e_{a,b;i,j}^{(r)},e_{c,d;k,l}^{(s)}]=\delta_{b,c}\delta_{k,j}e_{a,d;i,l}^{(r+s)}-\delta_{d,a}\delta_{l,i}e_{c,b;k,j}^{(r+s)}
\end{align}
for $1\leq a<b\leq m$, $1\leq c<d\leq m$, $r,s\geq 0$ and all admissible $i,j,k,l$ (\cite[Lemma 6.7]{BK05}).
Then using \eqref{gryn+relation} one obtains that $\widehat{Y}_n$ is spanned as a vector space by the monomials in the elements $\{D_{a;i,j}^{(r)}, E_{a,b;i,j}^{(r)},F_{b,a;j,i}^{(r)}\}$ taken in some fixed order. By Lemma \ref{Lemma:chi iso},
we can identify the graded algebra $\gr Y_n$ with $U(\fg)$.
We apply the identification in conjunction with the PBW theorem for $U(\fg)$ to see that the images of the monomials in the elements $\{D_{a;i,j}^{(r)}, E_{a,b;i,j}^{(r)},F_{b,a;j,i}^{(r)}\}$ under $\theta$ are linearly independent (cf. \cite[Lemma 5.10]{BK05} and \cite[Proposition 8.4]{Peng11}).
It follows that $\theta$ is injective.
\end{proof}

\begin{Remark}
Note that our characteristic $p>2$.
Set $v:=u$ in \eqref{cubic serre 2ge relation}, then replace $w$ by $v$, we get 
\begin{align}\label{cubic serre 1ge relation}
[E_{a;i,j}(u),[E_{a;k,l}(u),E_{b;g,h}(v)]]=0 \quad\:\:\text{ if }|a-b|=1.
\end{align}
This can be obtained from \cite[Lemma 6.5]{BK05} in conjunction with the shift map.
Taking the $u^{-2r}v^{-t}$-coefficient in \eqref{cubic serre 1ge relation} and using also \eqref{pr13} gives
\begin{align}\label{cubic serre 1ge relation-coeff}
  [E_{a;i,j}^{(r)}, [E_{a;k,l}^{(r)}, E_{b;f,g}^{(t)}]]=0
\quad\:\:\text{ if }|a-b|=1.
\end{align}
\end{Remark}

Recall that by Lemma \ref{Lemma:chi iso}, 
we may identify $e_{i,j}t^r$ with $\gr_r t_{i,j}^{(r+1)}$. 
Using \eqref{quasiD}-\eqref{quasiF}, one sees that $D_{a;i,j}^{(r+1)}$, $E_{a,b;i,j}^{(r+1)}$ and $F_{b,a;i,j}^{(r+1)}$ are all belong to ${\rm F}_r Y_n$,
and under our identification we have that 
\begin{equation}\label{identification}
e_{a,b;i,j}t^r=e_{p_a(\mu)+i,p_b(\mu)+j}t^{r}
=\left\{
\begin{array}{ll}
\gr_{r}D_{a;i,j}^{(r+1)}&\text{if $a=b$,}\\
\gr_{r}E_{a,b;i,j}^{(r+1)}&\text{if $a<b$,}\\
\gr_{r}F_{a,b;i,j}^{(r+1)}&\text{if $a>b$.}
\end{array}\right.
\end{equation}

Using the PBW theorem for $U(\fg)$,
we obtain the PBW basis for $Y_n$ in terms of the parabolic generators (cf. \cite[Theorem B]{BK05}).
\begin{Theorem}\label{Theorem:PBW parabolic generators}
Let 
\begin{align*}
I&:=\{D_{a;i,j}^{(r)};~1\leq a\leq m, 1\leq i,j\leq \mu_a, r>0\},\\
J&:=\{E_{a,b;i,j}^{(r)};~1\leq a<b\leq m, 1\leq i\leq \mu_a, 1\leq j\leq \mu_b, r>0\},\\
K&:=\{F_{b,a;i,j}^{(r)};~1\leq a<b\leq m, 1\leq i\leq \mu_b, 1\leq j\leq \mu_a, r>0\}.
\end{align*}
Then the set of all monomials in the union of the sets $I,J$ and $K$ taken in some fixed order forms a basis for $Y_n$.
\end{Theorem}

We will need the power series forms of some of the relations from Theorem \ref{Theorem:parabolic presentation}.
\begin{Lemma}
The following identities hold in $Y_n[[u^{-1},v^{-1}]]$ for all admissible $i,j,k,l$:
\begin{align}
(u-v)[E_{a;i,j}(u), E_{a;k,l}(v)] &=\big(E_{a;i,l}(v)-E_{a;i,l}(u)\big)\big(E_{a;k,j}(v)-E_{a;k,j}(u)\big),\label{ee relation}\\
&=\big(E_{a;k,j}(u)-E_{a;k,j}(v)\big)\big(E_{a;i,l}(u)-E_{a;i,l}(v)\big),\label{ee relation2}\\
(u-v)[E_{a;i,j}(u), F_{a;k,l}(v)] &=D'_{a;i,l}(u)D_{a+1;k,j}(u)-D'_{a;i,l}(v)D_{a+1;k,j}(v),\label{ef relation}\\
(u-v)[E_{a;i,j}(u), D_{a;k,l}(v)] &=D_{a;k,\alpha}(v)\big(E_{a;\alpha,j}(u)-E_{a;\alpha,j}(v)\big)\delta_{i,l},\label{ed1 relation}\\
(u-v)[E_{a;i,j}(u), D'_{a+1;k,l}(v)] &=\big(E_{a;i,\beta}(u)-E_{a;i,\beta}(v)\big)D'_{a+1;\beta,l}(v)\delta_{k,j},\label{ed2' relation}\\
(u-v)[E_{a;i,j}(u), D_{a+1;k,l}(v)] &=D_{a+1;k,j}(v)\big(E_{a;i,l}(v)-E_{a;i,l}(u)\big),\label{ed2 relation}\\
(u-v)[E_{a;i,j}(u), D'_{a;k,l}(v)] &=\big(E_{a;k,j}(v)-E_{a;k,j}(u)\big)D'_{a;i,l}(v).\label{ed1' relation}
\end{align}
(Here, $\alpha$ resp. $\beta$ should be summed over $1,\dots,\mu_a$ resp. $1,\dots,\mu_{a+1}$)
\end{Lemma}
\begin{proof}
Equations \eqref{ee relation} and \eqref{ef relation}-\eqref{ed2 relation} were proven over $\CC$ in \cite[Lemma 6.3]{BK05} and \cite[Theorem 2]{Peng11}, the same proof works here.
As $[E_{a;i,j}(u), E_{a;k,l}(v)]=-[E_{a;k,l}(v),E_{a;i,j}(u)]$,
\eqref{ee relation2} follows easily from \eqref{ee relation}.
For \eqref{ed1' relation}, we have by \eqref{ed1 relation} that 
\[
(u-v)\big(D_{a;i,j}(u)E_{a;k,l}(v)-E_{a;k,l}(v)D_{a;i,j}(u)\big)=\sum\limits_{\alpha}D_{a;i,\alpha}(u)\big(E_{a;\alpha,l}(v)-E_{a;\alpha,l}(u)\big)\delta_{k,j}.
\]
Multiplying $D'_{a;j,h}(u)$ on the right side yields 
\begin{align*}
&(u-v)\sum\limits_{j}\big(D_{a;i,j}(u)E_{a;k,l}(v)-E_{a;k,l}(v)D_{a;i,j}(u)\big)D'_{a;j,h}(u)\\
=&\sum\limits_{j,\alpha}D_{a;i,\alpha}(u)\big(E_{a;\alpha,l}(v)-E_{a;\alpha,l}(u)\big)D'_{a;j,h}(u)\delta_{k,j}.
\end{align*}
Using the fact that $D'_{a}(u)=(D_a(u))^{-1}$, we obtain
\begin{align*}
&(u-v)\big(\sum\limits_{j}D_{a;i,j}(u)E_{a;k,l}(v)D'_{a;j,h}(u)-E_{a;k,l}(v)\delta_{i,h}\big)\\
=&\sum\limits_{\alpha}D_{a;i,\alpha}(u)\big(E_{a;\alpha,l}(v)-E_{a;\alpha,l}(u)\big)D'_{a;k,h}(u).
\end{align*}
Multiplying $D'_{a;s,i}(u)$ on the left side gives the identity
\begin{align*}
&(u-v)\big(\sum\limits_{i,j}D'_{a;s,i}(u)D_{a;i,j}(u)E_{a;k,l}(v)D'_{a;j,h}(u)-\sum\limits_{i}D'_{a;s,i}(u)E_{a;k,l}(v)\delta_{i,h}\big)\\
=&\sum\limits_{i,\alpha}D'_{a;s,i}(u)D_{a;i,\alpha}(u)\big(E_{a;\alpha,l}(v)-E_{a;\alpha,l}(u)\big)D'_{a;k,h}(u).
\end{align*}
Using again the fact that $D'_{a}(u)=(D_a(u))^{-1}$ and simplifying the result, we have
\begin{align*}
&(u-v)\big(E_{a;k,l}(v)D'_{a;s,h}(u)-D'_{a;s,h}(u)E_{a;k,l}(v)\big)\\
=&\big(E_{a;s,l}(v)-E_{a;s,l}(u)\big)D'_{a;k,h}(u).
\end{align*}
Then replacing $(k,l,s,h)$ by $(i,j,k,l)$ gives the desired identity.
\end{proof}

\begin{Corollary}
The following hold in $Y_n[[u^{-1}]]$:
\begin{align}
D_{a;i,j}(u)E_{a-1;k,j}(u)&=E_{a-1;k,j}(u+1)D_{a;i,j}(u)\label{DE=Eu+1D},\\
D_{a;i,j}(u)E_{a;j,l}(u)-E_{a;j,l}(u-1)D_{a;i,j}(u)&=\sum\limits_{\alpha\neq j}D_{a;i,\alpha}(u)\big(E_{a;\alpha,l}(u-1)-E_{a;\alpha,l}(u)\big).\label{Eu-1D=DE}
\end{align}
\end{Corollary}
\begin{proof}
These follow from \eqref{ed1 relation} and \eqref{ed2 relation}.
For example, to get \eqref{Eu-1D=DE}, set $v:=u+1$ in \eqref{ed1 relation}, simplify, we have
\[
D_{a;k,i}(u+1)E_{a;i,j}(u+1)-E_{a;i,j}(u)D_{a;k,i}(u+1)=\sum\limits_{\alpha\neq i}D_{a;k,\alpha}(u+1)\big(E_{a;\alpha,j}(u)-E_{a;\alpha,j}(u+1)\big).
\]
Replacing $u$ and $(i,j,k)$ by $u-1$ and $(j,l,i)$ gives the desired relation.
\end{proof}

\begin{Remark}\label{Remark:svsu/u-v}
Observe that for any formal series $s(u)=\sum_{r\geq 0}s^{(r)}u^{-r}$ we have the identity 
\begin{align}
\frac{s(v)-s(u)}{u-v}=\sum\limits_{r,s\geq 1}s^{(r+s-1)}u^{-r}v^{-s}.
\end{align}
Dividing by $u-v$ on both sides of \eqref{ed1' relation} and considering the coefficients of $u^{-r}v^{-s}$, we have 
\begin{align}\label{eijrda'kls}
[E_{a;i,j}^{(r)},D'^{(s)}_{a;k,l}]=\sum\limits_{t=0}^{s-1}E_{a;k,j}^{(r+s-1-t)}D'^{(t)}_{a;i,l}.
\end{align}
\end{Remark}

The following Lemma is a generalization of \cite[Lemma 4.8]{BT18}.
\begin{Lemma}\label{Lemma:EDE relations1}
The following identities hold in $Y_n[[u^{-1},v^{-1}]]$ for all admissible $i,j,k,l$ and all $\ell\geq 0$.  
  \begin{align}
&(u-v)[E_{a;i,j}(u), \big(E_{a;i,l}(v)-E_{a;i,l}(u)\big)\big(E_{a;i,j}(v)-E_{a;i,j}(u)\big)^{\ell}\big(E_{a;k,j}(v)-E_{a;k,j}(u)\big)]\label{eee relation}\\
&=(\ell+2)\big(E_{a;i,l}(v)-E_{a;i,l}(u)\big)\big(E_{a;i,j}(v)-E_{a;i,j}(u)\big)^{\ell+1}\big(E_{a;k,j}(v)-E_{a;k,j}(u)\big),\nonumber\\
&(u-v)[E_{a;i,j}(u), D_{a;k,\alpha}(v)\big(E_{a;\alpha,j}(v)-E_{a;\alpha,j}(u)\big)\big(E_{a;i,j}(v)-E_{a;i,j}(u)\big)^{\ell}]\label{ede relation}\\
&=\ell D_{a;k,\alpha}(v)\big(E_{a;\alpha,j}(v)-E_{a;\alpha,j}(u)\big)\big(E_{a;i,j}(v)-E_{a;i,j}(u)\big)^{\ell+1},\nonumber\\
&(u-v)[E_{a;i,j}(u), D_{a+1;k,j}(v)\big(E_{a;i,j}(v)-E_{a;i,j}(u)\big)^{\ell}\big(E_{a;i,l}(v)-E_{a;i,l}(u)\big)]\label{ed2e relation}\\
&=(\ell+2)D_{a+1;k,j}(v)\big(E_{a;i,j}(v)-E_{a;i,j}(u)\big)^{\ell+1}\big(E_{a;i,l}(v)-E_{a;i,l}(u)\big),\nonumber\\
&(u-v)[E_{a;i,j}(u), D_{a+1;k,j}(v)\big(E_{a;i,j}(v)-E_{a;i,j}(u)\big)^{\ell}D'_{a;i,l}(v)]\label{ed2ed' relation}\\
&=(\ell+2)D_{a+1;k,j}(v)\big(E_{a;i,j}(v)-E_{a;i,j}(u)\big)^{\ell+1}D'_{a;i,l}(v).\nonumber
\end{align}
(Here, $\alpha$ should be summed over $1,\dots,\mu_a$.)
\end{Lemma}
\begin{proof}
The relation \eqref{eee relation} follows from \eqref{ee relation} and the Leibniz rule. We just check for the case $\ell=0$.
In view of \eqref{ee relation}, we have 
\begin{align*}
&(u-v)[E_{a;i,j}(u), \big(E_{a;i,l}(v)-E_{a;i,l}(u)\big)]=(u-v)[E_{a;i,j}(u), E_{a;i,l}(v)]\\
&=\big(E_{a;i,l}(v)-E_{a;i,l}(u)\big)\big(E_{a;i,j}(v)-E_{a;i,j}(u)\big)    
\end{align*}
and
\begin{align*}
 (u-v)[E_{a;i,j}(u), \big(E_{a;k,j}(v)-E_{a;k,j}(u)\big)]=\big(E_{a;i,j}(v)-E_{a;i,j}(u)\big)\big(E_{a;k,j}(v)-E_{a;k,j}(u)\big).
\end{align*}
So that the Leibniz rule implies that
 \begin{align*}
&(u-v)[E_{a;i,j}(u), \big(E_{a;i,l}(v)-E_{a;i,l}(u)\big)\big(E_{a;k,j}(v)-E_{a;k,j}(u)\big)]\\
&=2\big(E_{a;i,l}(v)-E_{a;i,l}(u)\big)\big(E_{a;i,j}(v)-E_{a;i,j}(u)\big)\big(E_{a;k,j}(v)-E_{a;k,j}(u)\big).
\end{align*}

For \eqref{ede relation},
by the Leibniz rule, together with \eqref{ee relation},
it suffices to show that
\begin{align}\label{ede=0}
(u-v)[E_{a;i,j}(u), \sum\limits_{\alpha}D_{a;k,\alpha}(v)\big(E_{a;\alpha,j}(v)-E_{a;\alpha,j}(u)\big)]=0.
\end{align}
By \eqref{ed1 relation} and \eqref{ee relation}, we have
\begin{align*}
&(u-v)[E_{a;i,j}(u), \sum\limits_{\alpha} D_{a;k,\alpha}(v)\big(E_{a;\alpha,j}(v)-E_{a;\alpha,j}(u)\big)]\\
=&(u-v)\sum\limits_{\alpha}[E_{a;i,j}(u), D_{a;k,\alpha}(v)]\big(E_{a;\alpha,j}(v)-E_{a;\alpha,j}(u)\big)\\
&+\sum\limits_{\alpha}D_{a;k,\alpha}(v)(u-v)[E_{a;i,j}(u),E_{a;\alpha,j}(v)]\\
=&\sum\limits_{\alpha}D_{a;k,\alpha}(v)\big(E_{a;\alpha,j}(u)-E_{a;\alpha,j}(v)\big)\big(E_{a;i,j}(v)-E_{a;i,j}(u)\big)\\
&+\sum\limits_{\alpha}D_{a;k,\alpha}(v)\big(E_{a;i,j}(v)-E_{a;i,j}(u)\big)\big(E_{a;\alpha,j}(v)-E_{a;\alpha,j}(u)\big)\\
=&0.
\end{align*}
The last equality follows from \eqref{ee relation2},
we thus obtain \eqref{ede=0}.

The relation \eqref{ed2e relation}-\eqref{ed2ed' relation} can be proved by similar methods.
They follow from \eqref{ee relation},
\eqref{ee relation2}, \eqref{ed2 relation} and \eqref{ed1' relation} using Leibniz again.
\end{proof}
The following relations are closely related to the ones in Lemma \ref{Lemma:EDE relations1}.
\begin{Lemma}\label{Lemma:ErDrEr relations2}
For any $a=1,\dots,m-1$, $\ell\geq 0, r, s>0$ and admissible $i,j,k,l$, we have that
\begin{align}\label{1111 relation}
\Bigg[E_{a;i,j}^{(r)}, \hspace{-10mm}\sum_{\substack{s_1,s_2,t_1,\dots,t_{\ell} \geq r \\ s_1+s_2+t_1+\cdots+t_{\ell}=
    (\ell+1)(r-1)+s}}
\hspace{-10mm}
E_{a;i,l}^{(s_1)}E_{a;i,j}^{(t_1)}
\cdots
E_{a;i,j}^{(t_{\ell})}E_{a;k,j}^{(s_{2})}\Bigg]
&=(\ell+2)
\hspace{-15mm}\sum_{\substack{s_1,s_2,t_1,\dots,t_{\ell+1} \geq r \\ s_1+s_2+t_1+\cdots+t_{\ell+1} =
    (\ell+2)(r-1)+s}}
\hspace{-15mm}E_{a;i,l}^{(s_1)}E_{a;i,j}^{(t_1)}
\cdots
E_{a;i,j}^{(t_{\ell+1})}E_{a;k,j}^{(s_2)},
\end{align}
\begin{align}\label{2222 relation}
\Bigg[E_{a;i,j}^{(r)}, \hspace{-10mm}\sum_{\substack{s_1,s_2,t_1,\dots,t_{\ell} \leq r-1 \\ s_1+s_2+t_1+\cdots+t_{\ell}=
    (\ell+1)(r-1)+s}}
\hspace{-10mm}
E_{a;i,l}^{(s_1)}E_{a;i,j}^{(t_1)}
\cdots
E_{a;i,j}^{(t_{\ell})}E_{a;k,j}^{(s_{2})}\Bigg]
&=-(\ell+2)
\hspace{-17mm}\sum_{\substack{s_1,s_2,t_1,\dots,t_{\ell+1} \leq r-1 \\ s_1+s_2+t_1+\cdots+t_{\ell+1} =
-(\ell+2)(r-1)+s}}
\hspace{-15mm}E_{a;i,l}^{(s_1)}E_{a;i,j}^{(t_1)}
\cdots
E_{a;i,j}^{(t_{\ell+1})}E_{a;k,j}^{(s_2)},
\end{align}
\begin{align}\label{3333 raltion}
\Bigg[E_{a;i,j}^{(r)}, \hspace{-10mm}\sum_{\substack{s_1,t_1,\dots,t_{\ell} \geq r, t \geq 0 \\ s_1+t_1+\cdots+t_{\ell}+t =
    (\ell+1)(r-1)+s}}\hspace{-10mm}
D_{a;k,\alpha}^{(t)}E_{a;\alpha,j}^{(s_1)}E_{a;i,j}^{(t_1)}
\cdots
E_{a;i,j}^{(t_{\ell})}\Bigg]
&=\ell\hspace{-15mm}
\sum_{\substack{s_1,t_1,\dots,t_{\ell+1} \geq r, t \geq 0 \\ s_1+t_1+\cdots+t_{\ell+1}+t =(\ell+2)(r-1)+s}}
\hspace{-15mm}D_{a;k,\alpha}^{(t)}E_{a;\alpha,j}^{(s_1)}E_{a;i,j}^{(t_1)}
\cdots
E_{a;i,j}^{(t_{\ell+1})},
\end{align}
\begin{align}\label{4444 raltion}
\Bigg[E_{a;i,j}^{(r)}, \hspace{-10mm}\sum_{\substack{s_1,t_1,\dots,t_{\ell} \geq r, t \geq 0 \\ s_1+t_1+\cdots+t_{\ell}+t =
    (\ell+1)(r-1)+s}}\hspace{-10mm}
D_{a+1;k,j}^{(t)}E_{a;i,j}^{(t_1)}\cdots E_{a;i,j}^{(t_\ell)}
E_{a;i,l}^{(s_1)}\Bigg]
&=(\ell+2)\hspace{-15mm}
\sum_{\substack{s_1,t_1,\dots,t_{\ell+1} \geq r, t \geq 0 \\ s_1+t_1+\cdots+t_{\ell+1}+t =(\ell+2)(r-1)+s}}
\hspace{-15mm}D_{a+1;k,j}^{(t)}E_{a;i,j}^{(t_1)}\cdots E_{a;i,j}^{(t_{\ell+1})}
E_{a;i,l}^{(s_1)},
\end{align}
\begin{align}\label{5555 raltion}
\Bigg[E_{a;i,j}^{(r)}, \hspace{-10mm}\sum_{\substack{t_1,\dots,t_{\ell} \geq r, t, u\geq 0 \\ t_1+\cdots+t_{\ell}+t+u =
\ell(r-1)+s}}\hspace{-10mm}
D_{a+1;k,j}^{(t)}E_{a;i,j}^{(t_1)}\cdots E_{a;i,j}^{(t_\ell)}
D'^{(u)}_{a;i,l}\Bigg]
&=(\ell+2)\hspace{-15mm}
\sum_{\substack{t_1,\dots,t_{\ell+1} \geq r, t, u\geq 0 \\
t_1+\cdots+t_{\ell+1}+t+u =
(\ell+1)(r-1)+s}}
\hspace{-15mm}D_{a+1;k,j}^{(t)}E_{a;i,j}^{(t_1)}\cdots E_{a;i,j}^{(t_{\ell+1})}
D'^{(u)}_{a;i,l}.
\end{align}
(Here, $\alpha$ should be summed over $1,\dots,\mu_a$.)
\end{Lemma}
\begin{proof}
They can be proved by a very similar methods as in the proof of Lemma \ref{Lemma:EDE relations1}.
For \eqref{1111 relation}, we have by \eqref{pr7} that
\[
[E_{a;i,j}^{(r)}, E_{a;i,l}^{(s_1)}]=\sum_{\substack{s_1',s_1'' \geq r \\
s'_1+s_1''=(r-1)+s_1}} E_{a;i,l}^{(s_1')}E_{a;i,j}^{(s_1'')},~
[E_{a;i,j}^{(r)}, E_{a;k,j}^{(s_2)}]=\sum_{\substack{s_2',s_2'' \geq r \\
s'_2+s_2''=(r-1)+s_2}} E_{a;i,j}^{(s_2')}E_{a;k,j}^{(s_2'')}
\]
and
\[
[E_{a;i,j}^{(r)}, E_{a;i,j}^{(t_q)}]=\sum_{\substack{t_q',t_q'' \geq r \\
t'_q+t_q''=(r-1)+t_q}} E_{a;i,j}^{(t_q')}E_{a;i,j}^{(t_q'')}
\]
for $0<r\leq s_1,s_2,t_q$.
Using the above and the Leibniz rule, we deduce that the left side of \eqref{1111 relation} equals
\begin{align*}
&\sum_{\substack{s_1,s_2,t_1,\dots,t_{\ell} \geq r \\ s_1+s_2+t_1+\cdots+t_{\ell}=
    (\ell+1)(r-1)+s}}
\hspace{2mm}\sum_{\substack{s_1',s_1'' \geq r \\
s'_1+s_1''=(r-1)+s_1}}
E_{a;i,l}^{(s_1')}E_{a;i,j}^{(s_1'')}E_{a;i,j}^{(t_1)}
\cdots
E_{a;i,j}^{(t_{\ell})}E_{a;k,j}^{(s_{2})}\\
+&\sum_{\substack{s_1,s_2,t_1,\dots,t_{\ell} \geq r \\ s_1+s_2+t_1+\cdots+t_{\ell}=
    (\ell+1)(r-1)+s}}
\hspace{2mm}\sum_{\substack{s_2',s_2'' \geq r \\
s'_2+s_2''=(r-1)+s_1}}
E_{a;i,l}^{(s_1)}E_{a;i,j}^{(t_1)}
\cdots
E_{a;i,j}^{(t_{\ell})}E_{a;i,j}^{(s_2')}E_{a;k,j}^{(s_2'')}\\
+&\sum\limits_{q=1}^{\ell}\sum_{\substack{s_1,s_2,t_1,\dots,t_{\ell} \geq r \\ s_1+s_2+t_1+\cdots+t_{\ell}=
    (\ell+1)(r-1)+s}}
\hspace{2mm}\sum_{\substack{t_q',t_q'' \geq r \\
t'_q+t_q''=(r-1)+t_q}}
E_{a;i,l}^{(s_1)}E_{a;i,j}^{(t_1)}\cdots E_{a;i,j}^{(t_{q-1})}
E_{a;i,j}^{(t_q')}E_{a;i,j}^{(t_q'')}E_{a;i,j}^{(t_{q+1})}\cdots E_{a;i,j}^{(t_{\ell})}E_{a;i,j}^{(s_2)}\\
=& 2\sum_{\substack{s_1,s_2,t_1,\dots,t_{\ell+1} \geq r \\ s_1+s_2+t_1+\cdots+t_{\ell+1} =
    (\ell+2)(r-1)+s}}
\hspace{-15mm}E_{a;i,l}^{(s_1)}E_{a;i,j}^{(t_1)}
\cdots
E_{a;i,j}^{(t_{\ell+1})}E_{a;k,j}^{(s_2)}\\
+&\sum\limits_{q=1}^{\ell}\sum_{\substack{s_1,s_2,t_1,\dots,t_{\ell+1} \geq r \\ s_1+s_2+t_1+\cdots+t_{\ell+1}=
    (\ell+2)(r-1)+s}}
E_{a;i,l}^{(s_1)}E_{a;i,j}^{(t_1)}\cdots E_{a;i,j}^{(t_{\ell+1})}E_{a;i,j}^{(s_2)},
\end{align*}
which gives the right hand side of \eqref{1111 relation}.
The proof of \eqref{2222 relation} is similar.
Now we recall the following relation
\[
\sum_{\substack{s_1,s_2 \geq r \\
s_1+s_2=(r-1)+s}}
E_{a;i,l}^{(s_1)}E_{a;k,j}^{(s_2)}=\sum_{\substack{s_1,s_2 \geq r \\
s_1+s_2=(r-1)+s}}
E_{a;k,j}^{(s_1)}E_{a;i,l}^{(s_2)},
\]
which also follows from \eqref{pr7}.
Then \eqref{3333 raltion} follows from this and \eqref{1111 relation} with the Leibniz rule.
Finally \eqref{4444 raltion} and \eqref{5555 raltion} follow similarly from \eqref{3333 raltion} together with \eqref{pr5} and \eqref{eijrda'kls}.
\end{proof}

We define
\begin{align}\label{definition:down formula}
D_{a;i,j\downarrow \ell}(u)&:=D_{a;i,j}(u)D_{a;i,j}(u-1)\cdots D_{a;i,j}(u-\ell+1),
\end{align}
\begin{align}\label{definition:up formula}
D_{a;i,j\uparrow \ell}(u)&:=D_{a;i,j}(u)D_{a;i,j}(u+1)\cdots D_{a;i,j}(u+\ell-1).
\end{align}
By convention, we set $D_{a;i,j\downarrow 0}(u)=D_{a;i,j\uparrow 0}(u)=1$.

Using \eqref{pr3} and induction on $r+s$, it is easy to see in particular that 
\begin{align}\label{Daijrs comm}
D_{a;i,j}^{(r)}D_{a;i,j}^{(s)}=D_{a;i,j}^{(s)}D_{a;i,j}^{(r)}.
\end{align}
Hence the order of the product on the right hand sides of \eqref{definition:down formula} and \eqref{definition:up formula} is irrelevant.
Moreover, \eqref{pr3} also implies that
\begin{align}\label{DD relatiaon}
(u-v)[D_{a;i,j}(u),D_{a;k,l}(v)]=\big(D_{a;k,j}(u)D_{a;i,l}(v)-D_{a;k,j}(v)D_{a;i,l}(u)\big).     
\end{align}

\begin{Lemma}\label{Lemma:DD relation induction}
The following relations hold: 
\begin{align}
D_{a;i,j}(u-1)D_{a;i,l}(u)&=D_{a;i,l}(u-1)D_{a;i,j}(u),\label{Dau-1 Da}\\
(\ell+1)D_{a;i,l}(u)D_{a;i,j}(u+\ell)&=\ell D_{a;i,j}(u+\ell)D_{a;i,l}(u)+D_{a;i,l}(u+\ell)D_{a;i,j}(u)~\text{for all}~\ell\geq 0. \label{DD relation induction}  
\end{align}
\end{Lemma}
\begin{proof}
These follow from the relation \eqref{DD relatiaon}.
For example, to get \eqref{DD relation induction}, 
replace $(i,j,k,l)$ by $(i,l,i,j)$ and set $v:=u+\ell$ in \eqref{DD relatiaon}, then simplify it. 
\end{proof}

\begin{Lemma}\label{Lemma: Daij-up-down}
The following relations hold for all $\ell\geq 1$:
\begin{align}
(u-v)[D_{a;i,j\downarrow \ell}(u), E_{a;k,l}(v)]&=\delta_{j,k}\ell D_{a;i,j\downarrow \ell-1}(u-1)\sum\limits_{\alpha=1}^{\mu_a}D_{a;i,\alpha}(u)\big(E_{a;\alpha,l}(v)-E_{a;\alpha,l}(u)\big),\label{down formula}\\
(u-v)[D_{a;i,j\uparrow \ell}(u), E_{a-1;k,l}(v)]&=\ell D_{a;i,l}(u)D_{a;i,j\uparrow \ell-1}(u+1)\big(E_{a-1;k,j}(u)-E_{a-1;k,j}(v)\big).\label{up formula}
\end{align}
\end{Lemma}
\begin{proof}
It is clear that \eqref{ed1 relation} implies that \eqref{down formula} holds for $j\neq k$,
so that for \eqref{down formula}, we may assume that $k=j$.
We actually prove it in the following equivalent form:
\begin{align}\label{down formula equivalent}
(u-v-\ell)D_{a;i,j\downarrow \ell}(u)E_{a;j,l}(v)&=(u-v)E_{a;j,l}(v)D_{a;i,j\downarrow \ell}(u)-\ell D_{a;i,j\downarrow \ell}(u)E_{a;j,l}(u)\nonumber\\
&+\ell D_{a;i,j\downarrow \ell-1}(u-1)\sum\limits_{\alpha\neq j}D_{a;i,\alpha}(u)\big(E_{a;\alpha,l}(v)-E_{a;\alpha,l}(u)\big).
\end{align}
This follows when $\ell=1$ from \eqref{ed1 relation}.
To prove \eqref{down formula equivalent} in general, we proceed by induction on $\ell$.
Given \eqref{down formula equivalent} for some $\ell\geq 1$,
multiply both sides on the left by $(u-v-\ell-1)D_{a;i,j}(u-\ell)$ and use \eqref{Daijrs comm} to deduce that:
\begin{align}\label{left multi down formula equivalent}
&(u-v-\ell-1)(u-v-\ell)D_{a;i,j\downarrow \ell+1}(u)E_{a;j,l}(v)\nonumber\\
=&(u-v)(u-v-\ell-1)D_{a;i,j}(u-\ell)E_{a;j,l}(v)D_{a;i,j\downarrow \ell}(u)-\ell(u-v-\ell-1) D_{a;i,j\downarrow \ell+1}(u)E_{a;j,l}(u)\nonumber\\
+&\ell(u-v-\ell-1) D_{a;i,j\downarrow \ell}(u-1)\sum\limits_{\alpha\neq j}D_{a;i,\alpha}(u)\big(E_{a;\alpha,l}(v)-E_{a;\alpha,l}(u)\big).
\end{align}
Using the case of $\ell=1$ in \eqref{down formula equivalent} and replacing $u$ by $u-\ell$ give that
\begin{align}\label{use formula}
 (u-v-\ell-1)D_{a;i,j}(u-\ell)E_{a;j,l}(v)&=(u-v-\ell)E_{a;j,l}(v)D_{a;i,j}(u-\ell)-D_{a;i,j}(u-\ell)E_{a;j,l}(u-\ell)\nonumber\\
&+\sum\limits_{\alpha\neq j}D_{a;i,\alpha}(u-\ell)\big(E_{a;\alpha,l}(v)-E_{a;\alpha,l}(u-\ell)\big).   
\end{align}
We can also set $v=u-\ell$ in \eqref{down formula equivalent}.
If $p\nmid\ell$, then we have
\begin{align}\label{revised identity}
 E_{a;j,l}(u-\ell)D_{a;i,j\downarrow\ell}(u)\!=\!D_{a;i,j\downarrow\ell}(u)E_{a;j,l}(u)\!-\!D_{a;i,j\downarrow \ell-1}(u-1)\sum\limits_{\alpha\neq j}D_{a;i,\alpha}(u)\big(E_{a;\alpha,l}(u-\ell)\!-\!E_{a;\alpha,l}(u)\big). 
\end{align}
If $p\mid \ell$, then \eqref{down formula equivalent} yields
$E_{a;j,l}(v)D_{a;i,j\downarrow \ell}(u)=D_{a;i,j\downarrow \ell}(u)E_{a;j,l}(v)$,
so that \eqref{revised identity} also holds for this case.

Then substituting \eqref{use formula} into \eqref{left multi down formula equivalent} and using \eqref{revised identity},
we obtain
\begin{align*}
 &(u-v-\ell-1)(u-v-\ell)D_{a;i,j\downarrow \ell+1}(u)E_{a;j,l}(v)\\
=&(u-v)(u-v-\ell)E_{a;j,l}(v)D_{a;i,j\downarrow \ell+1}(u)-(u-v)D_{a;i,j}(u-\ell)E_{a;j,l}(u-\ell)D_{a;i,j\downarrow \ell}(u)\\
&+(u-v)\sum\limits_{\alpha\neq j}D_{a;i,\alpha}(u-\ell)\big(E_{a;\alpha,l}(v)-E_{a;\alpha,l}(u-\ell)\big)D_{a;i,j\downarrow \ell}(u)\\
&-\ell(u-v-\ell-1) D_{a;i,j\downarrow \ell+1}(u)E_{a;j,l}(u)\\
&+\ell(u-v-\ell-1) D_{a;i,j\downarrow \ell}(u-1)\sum\limits_{\alpha\neq j}D_{a;i,\alpha}(u)\big(E_{a;\alpha,l}(v)-E_{a;\alpha,l}(u)\big)\\
=&(u-v)(u-v-\ell)E_{a;j,l}(v)D_{a;i,j\downarrow \ell+1}(u)-(\ell+1)(u-v-\ell)D_{a;i,j\downarrow \ell+1}(u)E_{a;j,l}(u)\\
&+(u-v)D_{a;i,j\downarrow \ell}(u-1)\sum\limits_{\alpha\neq j}D_{a;i,\alpha}(u)\big(E_{a;\alpha,l}(u-\ell)-E_{a;\alpha,l}(u)\big)\\
&+(u-v)\sum\limits_{\alpha\neq j}D_{a;i,\alpha}(u-\ell)\big(E_{a;\alpha,l}(v)-E_{a;\alpha,l}(u-\ell)\big)D_{a;i,j\downarrow \ell}(u)\\
&+\ell(u-v-\ell-1) D_{a;i,j\downarrow \ell}(u-1)\sum\limits_{\alpha\neq j}D_{a;i,\alpha}(u)\big(E_{a;\alpha,l}(v)-E_{a;\alpha,l}(u)\big).\\
\end{align*}
Using \eqref{Dau-1 Da} and simplifying the above,
we obtain
\begin{align*}
 &(u-v-\ell-1)(u-v-\ell)D_{a;i,j\downarrow \ell+1}(u)E_{a;j,l}(v)\\
=&(u-v)(u-v-\ell)E_{a;j,l}(v)D_{a;i,j\downarrow \ell+1}(u)-(\ell+1)(u-v-\ell)D_{a;i,j\downarrow \ell+1}(u)E_{a;j,l}(u)\\
&+(\ell+1)(u-v-\ell)D_{a;i,j\downarrow \ell}(u-1)\sum\limits_{\alpha\neq j}D_{a;i,\alpha}(u)\big(E_{a;\alpha,l}(v)-E_{a;\alpha,l}(u)\big).
\end{align*}
Dividing both sides by $(u-v-\ell)$, 
we obtain \eqref{down formula equivalent} with $\ell$ replaced by $\ell+1$,
as required.

The proof of \eqref{up formula} is similar.
One can show equivalently that
\begin{align}\label{up formula equiv}
(u-v)D_{a;i,j\uparrow \ell}(u)E_{a-1;k,l}(v)=&(u-v)E_{a-1;k,l}(v)D_{a;i,j\uparrow \ell}(u)\nonumber\\
&+\ell D_{a;i,l}(u)D_{a;i,j\uparrow \ell-1}(u+1)\big(E_{a-1;k,j}(u)-E_{a-1;k,j}(v)\big). 
\end{align}
This follows when $\ell=1$ from \eqref{ed2 relation}.
Assume that \eqref{up formula equiv} holds for some $\ell\geq 1$.
Multiplying both sides on the left by $(u-v+\ell)D_{a;i,j}(u+\ell)$,
we obtain
\begin{align}\label{case2-1}
&(u-v)(u-v+\ell)D_{a;i,j\uparrow \ell+1}(u)E_{a-1;k,l}(v)\nonumber\\
=&(u-v)(u-v+\ell)D_{a;i,j}(u+\ell)E_{a-1;k,l}(v)D_{a;i,j\uparrow \ell}(u)\nonumber\\
+&\ell(u-v+\ell)D_{a;i,j}(u+\ell)D_{a;i,l}(u)D_{a;i,j\uparrow \ell-1}(u+1)\big(E_{a-1;k,j}(u)-E_{a-1;k,j}(v)\big).
\end{align}
Using the case of $\ell=1$ in \eqref{up formula equiv} and replacing $u$ by $u+\ell$ give that
\begin{align}\label{case2-2}
 (u-v+\ell)D_{a;i,j}(u+\ell)E_{a-1;k,l}(v)&=(u-v+\ell)E_{a-1;k,l}(v)D_{a;i,j}(u+\ell)\nonumber\\
&+D_{a;i,\ell}(u+\ell)\big(E_{a-1;k,j}(u+\ell)-E_{a-1;k,j}(v)\big).   
\end{align}
Moreover, the case of $l=j$ in \eqref{up formula equiv} yields
\begin{align}\label{case2-3}
(u-v)D_{a;i,j\uparrow \ell}(u)E_{a-1;k,j}(v)=&(u-v)E_{a-1;k,j}(v)D_{a;i,j\uparrow \ell}(u)\nonumber\\
&+\ell D_{a;i,j\uparrow \ell}(u)\big(E_{a-1;k,j}(u)-E_{a-1;k,j}(v)\big). 
\end{align}
Then substituting \eqref{case2-2} into \eqref{case2-1} and using 
\eqref{DE=Eu+1D}, \eqref{DD relation induction} and \eqref{case2-3},
we obtain
\begin{align}\label{case2-4}
(u-v)(u-v+\ell)&D_{a;i,j\uparrow \ell+1}(u)E_{a-1;k,l}(v)=(u-v)(u-v+\ell)E_{a-1;k,l}(v)D_{a;i,j\uparrow \ell+1}(u)\nonumber\\
+&(\ell+1)(u-v+\ell)D_{a;i,l}(u)D_{a;i,j\uparrow \ell}(u+1)\big(E_{a-1;k,j}(u)-E_{a-1;k,j}(v)\big).
\end{align}
Dividing both sides by $(u-v+\ell)$, 
we obtain \eqref{up formula equiv} with $\ell$ replaced by $\ell+1$,
as required.
\end{proof}

\subsection{The modular shifted Yangian}\label{subsection:shifted yangian}
Pick a shift matrix $\sigma$ and let $\mu=(\mu_1,\dots,\mu_m)$ be an admissible shape to $\sigma$ as in Section \ref{Section: shift matrix and shifted currentalgebra}.
The {\em shifted Yangian} is the subalgebra $Y_n(\sigma)\subseteq Y_n$ generated by the following elements:
\begin{align}\label{parabolic generator of shifted Yangian}
&\{D_{a;i,j}^{(r)};~1\leq a\leq m, 1\leq i,j\leq \mu_a, r>0\},\nonumber\\
&\{E_{a;i,j}^{(r)},F_{a;j,i}^{(s)};~1\leq a<m, 1\leq i\leq \mu_a, 1\leq j\leq \mu_{a+1}, r>s_{a,a+1}^{\mu}, s>s_{a+1,a}^{\mu}\}. 
\end{align}
Notice that when $\sigma$ is the zero matrix we have $Y_n(\sigma)=Y_n$ (see Theorem \ref{Theorem: yn gb parabolic generators}).
The elements $E_{a,b;i,j}^{(r)}\in Y_n$ do not lie in the subalgebra $Y_n(\sigma)$ for a general shift matrix.
Following \cite[(3.15)-(3.16)]{BK06},
we define elements ${^\sigma}E_{a,b,i,j}^{(r)}$ inductively by
\begin{align}\label{sigma Eabij}
{^\sigma}E_{a,b;i,j}^{(r)}:=[{^\sigma}E_{a,b-1;i,k}^{(r-s_{b-1,b}^{\mu})},E_{b-1;k,j}^{(s_{b-1,b}^{\mu}+1)}]
\end{align}
for $1\leq a<b\leq m$,
$1\leq i\leq \mu_a$, $1\leq j\leq \mu_b$ and $r>s_{a,b}^{\mu}$,
where $1\leq k\leq \mu_{b-1}$. 
The definition is independent of the choice of $k$, see for instance \cite[(6.9)]{BK05} and \cite[(3.15)]{BK06}.
Similarly, using the same indices except for $s>s_{b,a}^{\mu}$, we define ${^\sigma}F_{b,a,j,i}^{(s)}$ by
\begin{align}\label{sigma Fbaji}
{^\sigma}F_{b,a;j,i}^{(s)}:=[F_{b-1;j,k}^{(s_{b,b-1}^{\mu}+1)}, {^\sigma}F_{b-1,a;k,i}^{(s-s_{b,b-1}^{\mu})}].
\end{align}

\begin{Lemma}\label{lemma:grYnsigma=ugsigma}
For any shift matrix $\sigma$ and admissible shape $\mu$,
$1\leq a<b\leq m$, $1\leq i\leq \mu_a$, $1\leq j\leq\mu_b$ and $r\geq s_{a,b}^{\mu}$, $s\geq s_{b,a}^{\mu}$,
we have that ${^\sigma}E_{a,b;i,j}^{(r+1)}\in{\rm F}_r Y_n$ and ${^\sigma}F_{b,a,j,i}^{(s+1)}\in {\rm F}_s Y_n$.
Moreover, under the identification $\chi$ in Lemma \ref{Lemma:chi iso}, we have that
\begin{align}\label{gr sigama EF abij}
\gr_r{^\sigma}E_{a,b;i,j}^{(r+1)}=e_{a,b;i,j}t^r,\ \ \ \gr_r{^\sigma}F_{b,a;j,i}^{(s+1)}=e_{b,a;j,i}t^s.
\end{align}
Hence, $\gr Y_n(\sigma)$ is identified with the subalgebra $U(\fg_\sigma)$ of $U(\fg)$.
\end{Lemma}
\begin{proof}
We just prove the statements about ${^\sigma}E_{a,b;i,j}^{(r)}$.
When $b=a+1$, this follows from \eqref{identification}.
For general, we proceed by induction on $b-a$.
In view of \eqref{sigma Eabij},
the induction hypothesis implies that ${^\sigma}E_{a,b;i,j}^{(r+1)}\in {\rm F}_{r-s_{b-1,b}^{\mu}+s_{b-1,b}^{\mu}} Y_n$ as required.
Moreover, by induction again,
its image in the associated graded algebra is 
\[
[e_{a,b-1;i,k}t^{r-s_{b-1,b}^{\mu}},e_{b-1,b;k,j}t^{s_{b-1,b}^{\mu}}]=e_{a,b;i,j}t^r
\]
using \eqref{Lie bracket current lie 2}.
\end{proof}

From this and the PBW theorem for $U(\fg_\sigma)$, 
we also get the following PBW theorem for $Y_n(\sigma)$ (see also \cite[Theorem 3.2(iv)]{BK06}).
\begin{Theorem}\label{Theorem: PBW parabolic for Ynsigma}
Let 
\begin{align*}
I(\sigma)&:=I=\{D_{a;i,j}^{(r)};~1\leq a\leq m, 1\leq i,j\leq \mu_a, r>0\},\label{def:I }\\
J(\sigma)&:=\{{^\sigma}E_{a,b;i,j}^{(r)};~1\leq a<b\leq m, 1\leq i\leq \mu_a, 1\leq j\leq \mu_b, r>s_{a,b}^{\mu}\},\\
K(\sigma)&:=\{{^\sigma}F_{b,a;i,j}^{(r)};~1\leq a<b\leq m, 1\leq i\leq \mu_b, 1\leq j\leq \mu_a, r>s_{b,a}^{\mu}\}.
\end{align*}
Then the set of all monomials in the union of the sets $I(\sigma)$,
$J(\sigma)$ and $K(\sigma)$ taken in some fixed order forms a basis for $Y_n(\sigma)$.
\end{Theorem}    

We have defined $Y_n(\sigma)$ as a subalgebra of $Y_n$.
It can also be defined by generators and relations: 
the following theorem shows that it has its own parabolic presentation (cf. \cite[(3.3)-(3.14)]{BK06}).

\begin{Theorem}\label{Theorem:parabolic presentation of shifted Yangain}
The shifted Yangian $Y_n(\sigma)$ is generated by the elements \eqref{parabolic generator of shifted Yangian} 
subject to the relations \eqref{pr1}-\eqref{pr14},
interpreting admissible $a, b, f, g, k, l, r, s, t$ so that the left hand sides of these relations only involve generators of $Y_n(\sigma)$.
\end{Theorem}
\begin{proof}
The proof, 
which uses Theorem \ref{Theorem:parabolic presentation} and Lemma \ref{lemma:grYnsigma=ugsigma},
is similar to the proof of \cite[Theorem 4.15]{BT18}, 
and will be skipped here.
\end{proof}

\subsection{Automorphisms}\label{section: auto}
We list the following (anti)automorphisms of $Y_n$ which are needed in the next section;
see for instance \cite[Section 4.5]{BT18}.
\begin{enumerate}
\item(``Transposition'')
Let $\tau:Y_n\rightarrow Y_n$ be the anti-automorphism defined by $\tau(t_{i,j}^{(r)})=t_{j,i}^{(r)}$
or, on parabolic generators, $\tau(D_{a;i,j}(u))=D_{a;j,i}(u)$, $\tau(E_{a,b;i,j}(u))=F_{b,a;j,i}(u)$,
$\tau(F_{b,a;i,j}(u))=E_{a,b;j,i}(u)$ (cf. \cite[(6.6)-(6.8)]{BK05} and \cite[(3.20)]{BK06}).
\item (``Change of shift matrix'')
Suppose that $\sigma$ is a shift matrix with an admissible shape $\mu$ as usual and $\dot{\sigma}=(\dot{s}_{i,j})_{1\leq i,j\leq n}$
is another shift matrix satisfying $s_{i,i+1}+s_{i+1,i}=\dot{s}_{i,i+1}+\dot{s}_{i+1,i}$ for all $i=1,\dots,n-1$.
Clearly, $\mu$ is also admissible for $\dot{\sigma}$.
According to \cite[(3.21)]{BK06},
there is a unique algebra isomorphism $\iota: Y_n(\sigma)\Tilde{\rightarrow} Y_n(\dot{\sigma})$ defined by
\[
\iota(D_{a;i,j}^{(r)})=D_{a;i,j}^{(r)},
\iota({^{\sigma}}E_{a;i,j}^{(r)})={^{\dot{\sigma}}}E_{a;i,j}^{(r-s_{a,a+1}^{\mu}+\dot{s}_{a,a+1}^{\mu})},
\iota({^{\sigma}}F_{a;i,j}^{(r)})={^{\dot{\sigma}}}F_{a;i,j}^{(r-s_{a+1,a}^{\mu}+\dot{s}_{a+1,a}^{\mu})}.
\]
\item(``Permutation'')
Let $S_{n}$ be the Symmetric group on $n$ objects.
For each $w\in S_n$, there is an automorphism
$w:Y_n\rightarrow Y_n$ sending $t_{i,j}^{(r)} \mapsto t_{w(i),
  w(j)}^{(r)}$. This is clear from the RTT relation (\ref{RTT relations}).
\end{enumerate}

\begin{Lemma}\label{lemma:DEFab-DEFa}
Let $1\leq a<b\leq m$. Then the following statements hold:
\begin{enumerate}
    \item For any admissible $i,j$, if $i=j$, then the permutation automorphism of $Y_n$ by the transposition $(p_a(\mu)+1,p_a(\mu)+i)$ maps $D_{a;i,j}(u)\mapsto D_{a;1,1}(u)$, 
    and if $i\neq j$, then the permutation automorphism by $p_a(\mu)+i\mapsto p_a(\mu)+1$, $p_a(\mu)+j\mapsto p_a(\mu)+2$ maps $D_{a;i,j}(u)\mapsto D_{a;1,2}(u)$.
    \item For any admissible $i,j$, the permutation automorphism of $Y_n$ by the transposition $(p_{a+1}(\mu)+1,p_b(\mu)+j)$ maps
    $E_{a,b;i,j}(u)\mapsto E_{a;i,1}(u)$ and $F_{b,a;j,i}(u)\mapsto F_{a;1,i}(u)$.
\end{enumerate}
\end{Lemma}
\begin{proof}
This follows from \eqref{quasiD}-\eqref{quasiF}.  
\end{proof}

\section{Centers of $Y_n$ and $Y_n(\sigma)$}\label{section name:Centers}
Continue with a fixed shift matrix $\sigma$ and let $\mu=(\mu_1,\dots,\mu_m)$ be an admissible shape to $\sigma$.
In this section, we describe the center of the modular shifted Yangian by using parabolic generators. 
\subsection{Harish-Chandra center}\label{section name HC center}
Following \cite{BK05}, we define a power series $c(u)\in Y_n[[u^{-1}]]$ by the rule
\begin{align}\label{power series: HC center drinfeld}
c(u)=\sum\limits_{r\geq 0}c^{(r)}u^{-r}:=d_1(u)d_2(u-1)\cdots d_n(u-n+1).
\end{align}
According to \cite[Theorem 5.1]{BT18}, the {\em Harish-Chandra center} $Z_{\HC}(Y_n)$ is generated by the elements $\{c^{(r)};~r>0\}$ and $\gr_r c^{(r+1)}=z_r$, so that $c^{(1)}, c^{(2)},\dots$ are algebraically independent. 

For each $a=1,\dots,m$,
recall that $D_a(u)$ is a $\mu_a\times\mu_a$ matrix.
Following \cite{MNO96}, we define the {\em quantum determinant} of the matrix $D_a(u)$ as follows:
\begin{align}\label{power series:qdet of Da}
\qdet D_a(u):=\sum\limits_{\sigma\in S_{\mu_a}}\sgn(\sigma)D_{a;\sigma(1),1}(u)D_{a;\sigma(2),2}(u-1)\cdots D_{a;\sigma(\mu_a),\mu_a}(u-\mu_a+1),
\end{align}
where $S_{\mu_a}$ is the symmetric group on $\mu_a$ letters.

\begin{Proposition}\label{prop-HC center gb parabolic generators}
\begin{align}\label{prop:HC center gb parabolic generators}
 c(u)=\qdet D_1(u-p_1(\mu))\qdet D_2(u-p_2(\mu))\cdots\qdet D_m(u-p_m(\mu)).
\end{align}
In particular, the elements $c^{(r)}$ lie in the center of $Y_n(\sigma)$ and they are algebraically independent.
\end{Proposition}
\begin{proof}
Actually, in term of Drinfeld generators, we have
\[
\qdet D_a(u-p_a(u))=\prod\limits_{k=p_a(\mu)+1}^{p_{a+1}(\mu)}d_k(u-k+1).
\]
This has been proven over $\CC$ in \cite[Proposition 3.2]{CH23-1}, 
the same proof works here. 
Moreover, \eqref{prop:HC center gb parabolic generators} implies that the elements $c^{(r)}$ are also central elements in $Y_n(\sigma)$.
\end{proof}
Proposition \ref{prop-HC center gb parabolic generators} implies $Z_{\HC}(Y_n)\subseteq Z(Y_n(\sigma))$ and so we may also denote it $Z_{\HC}(Y_n(\sigma))$ and call it the {\em Harish-Chandra center} of $Y_n(\sigma)$.

\begin{Remark}
When $\mu=(1,\dots,1)$, the product in \eqref{prop:HC center gb parabolic generators} coincides with the definition of $c(u)$. 
When $\mu=(n)$, the product is just the usual quantum determinant of $T(u)$,
which is also equal to $c(u)$ (see \cite[Theorem 8.6]{BK05}).
\end{Remark}
\subsection{Off-diagonal $p$-central elements}
This subsection is a generalization of \cite[Section 5.2]{BT18}.
We investigate the $p$-central elements that lie in the root subalgebras
$Y^+_{a,b;i,j}, Y^-_{b,a;k,l}\subseteq Y_n$ for $1\leq a<b\leq m$ and all admissible $i,j,k,l$,
that is generated by $\{E_{a,b;i,j}^{(r)};~r>0\}$ and $\{F_{b,a;k,l}^{(r)};~r>0\}$,
respectively.

\begin{Lemma}\label{Lemma:EupFDE=0}
For $1\leq a<b\leq m$ and all admissible $i,j$,
all coefficients in the power series $(E_{a,b;i,j}(u))^p$ and $(F_{b,a;i,j}(u))^p$ belong to $Z(Y_n)$.
\end{Lemma}
\begin{proof}
Using Lemma \ref{lemma:DEFab-DEFa} and the anti-automorphism $\tau$, 
the proof reduces to checking that the coefficients of $(E_{a;i,j}(u))^p$ are central in $Y_n$ for each $a=1,\dots,m-1$ and admissible $i,j$.
Since we are in characteristic $p$,
Theorem \ref{Theorem: yn gb parabolic generators} implies that it is enough to show the following identities in $Y_n[[u^{-1},v^{-1}]]$ for all admissible $b,k,l$:
\begin{align}
(\ad E_{a;i,j}(u))^p (E_{b;k,l}(v)) &=0,\label{p power coff 1}\\
(\ad E_{a;i,j}(u))^p (D_{b;k,l}(v)) &=0,\label{p power coff 2}\\
(\ad E_{a;i,j}(u))^p (F_{b;k,l}(v)) &=0.\label{p power coff 3}
\end{align}

We first check \eqref{p power coff 1}.
We use \eqref{ee relation} and \eqref{eee relation} repeatedly:
\begin{align*}
&(u-v)^p(\ad E_{a;i,j}(u))^{p-1}([E_{a;i,j}(u), E_{a;k,l}(v)])\\
=&(u-v)^{p-1}(\ad E_{a;i,j}(u))^{p-1}\big((E_{a;i,l}(v)-E_{a;i,l}(u))(E_{a;k,j}(v)-E_{a;k,j}(u))\big)\\
=&(u-v)^{p-2}(\ad E_{a;i,j}(u))^{p-2}\big(2(E_{a;i,l}(v)-E_{a;i,l}(u))(E_{a;i,j}(v)-E_{a;i,j}(u))(E_{a;k,j}(v)-E_{a;k,j}(u))\big)\\
=&\cdots=p!\big((E_{a;i,l}(v)-E_{a;i,l}(u))(E_{a;i,j}(v)-E_{a;i,j}(u))^{p-1}(E_{a;k,j}(v)-E_{a;k,j}(u))\big)=0.
\end{align*}
Dividing by $(u-v)^p$ gives the desired identity.
To see that $(\ad E_{a;i,j}(u)))^p((E_{b;k,l})(v))=0$ when $|a-b|=1$, 
we actually already have that $(\ad E_{a;i,j}(u)))^2((E_{b;k,l})(v))=0$ by \eqref{cubic serre 1ge relation}.
When $|a-b|>1$, the identity is clear because $[E_{a;i,j}(u), E_{b;k,l}(v)]=0$ by \eqref{pr11} (see also \cite[Lemma 6.4(i)]{BK05}).

For \eqref{p power coff 2},
it is immediate from \eqref{pr5} if $b\neq a$ or $b\neq a+1$.
For the case $b=a$,
we have by \eqref{ed1 relation} and \eqref{ede=0} that
\begin{align*}
&(u-v)^2(\ad E_{a;i,j}(u))^{2}(D_{a;k,l}(v))\\
=&(u-v)(\ad E_{a;i,j}(u))(\sum\limits_{\alpha}D_{a;k,\alpha}(v)(E_{a;\alpha,j}(u)-E_{a;\alpha,j}(v))\delta_{i,l})=0.
\end{align*}
Hence, on dividing by $(u-v)^2$, we get $(\ad E_{a;i,j}(u))^p (D_{a;k,l}(v))=0$.
Finally, when $b=a+1$, repeated application of \eqref{ed2 relation} and \eqref{ed2e relation} yields
\begin{align*}
&(u-v)^p(\ad E_{a;i,j}(u))^{p}(D_{a+1;k,l}(v))\\
=&(u-v)^{p-1}(\ad E_{a;i,j}(u))^{p-1}\big(D_{a+1;k,j}(v)(E_{a;i,l}(v)-E_{a;i,l}(v))\big)\\
=&(u-v)^{p-2}(\ad E_{a;i,j}(u))^{p-2}\big(2 D_{a+1;k,j}(v)(E_{a;i,j}(v)-E_{a;i,j}(u))(E_{a;i,l}(v)-E_{a;i,l}(u))\big)\\
=&\cdots=p!D_{a+1;k,j}(v)(E_{a;i,j}(v)-E_{a;i,j}(u))^{p-1}(E_{a;i,l}(v)-E_{a;i,l}(u))=0.
\end{align*}
Dividing by $(u-v)^p$ completes the proof of \eqref{p power coff 2}.

Finally, for \eqref{p power coff 3}, when $b\neq a$, it follows from \eqref{pr4}.
When $a=b$, we have by \eqref{ed2ed' relation} that
\begin{align*}
&(u-v)^{p-1}(\ad E_{a;i,j}(u))^{p-1}(D_{a+1;k,j}(v)D'_{a;i,l}(v))\\
=&(u-v)^{p-2}(\ad E_{a;i,j}(u))^{p-2}\big(2 D_{a+1;k,j}(v)(E_{a;i,j}(v)-E_{a;i,j}(u))D'_{a;i,l}(v)\big)\\
=&\cdots=p!D_{a+1;k,j}(v)(E_{a;i,j}(v)-E_{a;i,j}(u))D'_{a;i,l}(v)=0.
\end{align*}
As a result, 
\[
(\ad E_{a;i,j}(u))^{p-1}(D_{a+1;k,j}(v)D'_{a;i,l}(v))=0.
\]
We also have 
\[
(\ad E_{a;i,j}(u))^{p-1}(D_{a+1;k,j}(u)D'_{a;i,l}(u))=0
\] 
by setting $v=u$. Then using \eqref{ef relation} and \eqref{pr3},
we conclude that
\begin{align*}
&(u-v)^p(\ad E_{a;i,j}(u))^p (F_{b;k,l}(v))\\
=&(u-v)^{p-1}(\ad E_{a;i,j}(u))^{p-1}(D_{a+1;k,j}(u)D'_{a;i,l}(u)-D_{a+1;k,j}(v)D'_{a;i,l}(v))=0.
\end{align*}
\end{proof}

\begin{Lemma}\label{Lemma:EabijrFbaijrFDE=0}
For $1\leq a<b\leq m$, $r>0$ and all admissible $i,j,k,l$,
we have that $(E_{a,b;i,j}^{(r)})^p$, $(F_{b,a;k,l}^{(r)})^p\in Z(Y_n)$.
\end{Lemma}
\begin{proof}
Similar to the proof of Lemma  \ref{Lemma:EupFDE=0}, 
it is enough to show that $(E_{a;i,j}^{(r)})^p\in Z(Y_n)$ for each $1\leq a\leq m-1$ and admissible $i,j,r$.
This reduces to checking 
\begin{align}
(\ad E_{a;i,j}^{(r)})^p (E_{b;k,l}^{(s)}) &=0,\label{coff p power coff 1}\\
(\ad E_{a;i,j}^{(r)})^p (D_{b;k,l}^{(s)}) &=0,\label{coff p power coff 2}\\
(\ad E_{a;i,j}^{(r)})^p (F_{b;k,l}^{(s)}) &=0.\label{coff p power coff 3}
\end{align} 
These may all be proved in a very similar way to \eqref{p power coff 1}-\eqref{p power coff 3}.
For example, we check \eqref{coff p power coff 2}.
Due to \eqref{pr5}, we may assume that $a\leq b\leq a+1$.
If $b=a$, then \eqref{pr5} and \eqref{3333 raltion} with $\ell=0$ yield
\begin{align*}
 (\ad E_{a;i,j}^{(r)})^2(D_{a;k,l}^{(s)})=-(\ad E_{a;i,j}^{(r)})(\sum_{\substack{s_1\geq r, t \geq 0 \\ s_1+t =(r-1)+s}}\sum\limits_{\alpha}D_{a;k,\alpha}^{(t)}E_{a;\alpha,j}^{(s_1)}\delta_{i,l})=0.
\end{align*}
When $b=a+1$, a consecutive application \eqref{pr5} and \eqref{4444 raltion} implies
\begin{align*}
(\ad E_{a;i,j}^{(r)})^p(D_{a+1;k,l}^{(s)})=&(\ad E_{a;i,j}^{(r)})^{p-1}\big(\sum_{\substack{s_1\geq r, t \geq 0 \\ s_1+t =(r-1)+s}}D_{a+1;k,j}^{(t)}E_{a;i,l}^{(s_1)}\big)\\
=&2(\ad E_{a;i,j}^{(r)})^{p-2}\big(\sum_{\substack{s_1, t_1\geq r, t \geq 0 \\ s_1+t_1+t =2(r-1)+s}}D_{a+1;k,j}^{(t)}E_{a;i,j}^{(t_1)}E_{a;i,l}^{(s_1)}\big)\\
=&\cdots = p!\sum_{\substack{s_1, t_1,\dots, t_{p-1}\geq r, t \geq 0 \\ s_1+t_1+\cdots+ t_{p-1}+t =p(r-1)+s}}D_{a+1;k,j}^{(t)}E_{a;i,j}^{(t_1)}\cdots E_{a;i,j}^{(t_{p-1})}E_{a;i,l}^{(s_1)}=0.
\end{align*}
This completes the proof of \eqref{coff p power coff 2}.
\end{proof}

Also, we put
\begin{align}\label{def: pabij-qabij}
P_{a,b;i,j}(u)=\sum\limits_{r\geq p}P_{a,b;i,j}^{(r)}u^{-r}:=E_{a,b;i,j}(u)^p,~Q_{b,a;i,j}(u)=\sum\limits_{r\geq p}Q_{b,a;i,j}^{(r)}u^{-r}:=F_{b,a;i,j}(u)^p.
\end{align}
\begin{Theorem}
For $1\leq a<b\leq m$ and all admissible $i,j,k,l$, 
the algebra $Z(Y_n)\cap Y^+_{a,b;i,j}$ and $Z(Y_n)\cap Y^-_{b,a;k,l}$ are infinite rank polynomial algebras freely generated by the central elements
$\{(E_{a,b;i,j}^{(r)})^p;~r>0\}$ and $\{(F_{b,a;k,l}^{(r)})^p;~r>0\}$, respectively.
We have that $(E_{a,b;i,j}^{(r)})^p$, $(F_{b,a;k,l}^{(r)})^p\in {\rm F}_{rp-p} Y_n$ and
\[
\gr_{rp-p}(E_{a,b;i,j}^{(r)})^p=(e_{a,b;i,j}t^{r-1})^p,~\gr_{rp-p}(F_{b,a;k,l}^{(r)})^p=(e_{b,a;k,l}t^{r-1})^p.
\]
For $r \geq p$ we have that
\begin{equation}\label{centre 1-222}
P_{a,b;i,j}^{(r)} = \left\{
\begin{array}{ll}
(E_{a,b;i,j}^{(r/p)})^p+(*)&\text{if $p \mid r$,}\\
(*)&\text{if $p \nmid r$,}
\end{array}
\right.
\end{equation}
where $(*)\in{\rm F}_{r-p-1} Y_n$
is a polynomial in
the elements
$(E_{a,b;i,j}^{(s)})^p$ for $1 \leq s<\lfloor r/p\rfloor$.
Hence, the central elements
$\{P_{a,b;i,j}^{(rp)};~r > 0\}$ give another algebraically
independent set of generators for
$Z(Y_n)\cap Y_{a,b;i,j}^+$
lifting the central elements $\{(e_{a,b;i,j}t^{r-1})^p;~r >
0\}$ of $\gr Y_n$.
Analogous statements with $Y_{a,b;i,j}^+, E, P$ and $e_{a,b;i,j}t^{r-1}$
replaced by $Y_{b,a;k,l}^-, F ,Q$ and $e_{b,a;k,l}t^{r-1}$ also hold.
\end{Theorem}
\begin{proof}
The proof, which uses Theorem \ref{theorem:center of ugsigma},
Lemma \ref{Lemma:EupFDE=0} and Lemma \ref{Lemma:EabijrFbaijrFDE=0},
is similar to the proof of \cite[Theorem 5.4]{BT18}.    
\end{proof}

\begin{Corollary}\label{corollary:central EF eleemnts ynsigma}
The elements $\{({^\sigma}E_{a,b;i,j}^{(r)})^p;~r>s_{a,b}^{\mu}\}$ and
$\{({^\sigma}F_{b,a;i,j}^{(r)})^p;~r>s_{b,a}^{\mu}\}$ are central in $Y_n(\sigma)$.
Moreover, they belong to ${\rm F}_{rp-p} Y_n(\sigma)$ and 
\begin{align}\label{gr of sigma EF abijrp}
\gr_{rp-p}({^\sigma}E_{a,b;i,j}^{(r)})^p=(e_{a,b;i,j}t^{r-1})^p,~\ \  \gr_{rp-p}({^\sigma}F_{b,a;i,j}^{(r)})^p=(e_{b,a;i,j}t^{r-1})^p.
\end{align}
\end{Corollary}
\begin{proof}
The assertion \eqref{gr of sigma EF abijrp} is immediate from \eqref{gr sigama EF abij},
so we just need to establish the centrality.
In case $\sigma$ is lower triangular, 
we have that ${^\sigma}E_{a,b;i,j}^{(r)}=E_{a,b;i,j}^{(r)}$ (see \eqref{EFabij r induction formula} and \eqref{sigma Eabij}).
Lemma \ref{Lemma:EabijrFbaijrFDE=0} implies that $({^\sigma}E_{a,b;i,j}^{(r)})^p$ is central in $Y_n$, so it is certainly central in the subalgebra $Y_n(\sigma)$.
For general case, we pick a lower triangular shift matrix $\dot{\sigma}$ such that 
$s_{i,i+1}+s_{i+1,i}=\dot{s}_{i,i+1}+\dot{s}_{i+1,i}$ for all $i=1,\dots,n-1$.
Using the change of shift matrix isomorphism $\iota$ from Section \ref{section: auto},
we obtain
\[
\iota(({^{\sigma}}E_{a,b;i,j}^{(r)})^p)=({^{\dot{\sigma}}}E_{a,b;i,j}^{(r-s_{a,b}^{\mu})})^p=(E_{a,b;i,j}^{(r-s_{a,b}^{\mu})})^p\in Z(Y_n(\dot{\sigma})).
\]
Consequently, $({^\sigma}E_{a,b;i,j}^{(r)})^p\in Z(Y_n(\sigma))$.
The centrality of $({^\sigma}F_{b,a;i,j}^{(r)})^p$ is proved similarly.
\end{proof}

\subsection{Diagonal $p$-central elements}
In this subsection we introduce the $p$-central elements generated by the diagonal parabolic generators $\{D_{a;i,j}^{(r)}\}$. 
For $a=1,\dots,m$ and admissible $i,j$, we define
\begin{align}\label{formula:baij}
B_{a;i,j}(u)=\sum\limits_{r\geq 0}B_{a;i,j}^{(r)}u^{-r}:=D_{a;i,j}(u)D_{a;i,j}(u-1)\cdots D_{a;i,j}(u-p+1).
\end{align}

\begin{Lemma}\label{Lemma:Baij in center}
For all $a=1,\dots,m$, $r>0$ and admissible $i,j$,
the element $B_{a;i,j}^{(r)}$ belongs to $Z(Y_n)$. 
\end{Lemma}
\begin{proof}
Thanks to Theorem \ref{Theorem: yn gb parabolic generators},
it suffices to check for all admissible $b,k,l$ that
\begin{align}\label{BEF=0}
[B_{a;i,j}(u),E_{b;k,l}(v)]&=0=[B_{a;i,j}(u),F_{b;k,l}(v)],
\end{align}
\begin{align}\label{BD=0}
[B_{a;i,j}(u),D_{b;k,l}(v)]&=0.
\end{align}

To prove \eqref{BEF=0}, we use the anti-automorphism $\tau$ to reduce to showing the first equality.
This is clear when $b\notin\{a-1,a\}$ by \eqref{pr5}.
Consider first the case $b=a-1$. Then \eqref{up formula} implies that
\[
[B_{a;i,j}(u),E_{a-1;k,l}(v)]=[D_{a;i,j\downarrow p}(u),E_{a-1;k,l}(v)]=[D_{a;i,j\uparrow p}(u-p+1),E_{a-1;k,l}(v)]=0.
\]
One argues similarly for the case $b=a$ using \eqref{down formula} instead of \eqref{up formula}.

For \eqref{BD=0}, when $b\neq a$, it follows from \eqref{pr3}.
If $b=a$, then we may assume that $a=1$ by using the shift map \eqref{psik D}.
In view of \eqref{quasiD}, it remains to show that
\[
[t_{i,j}(u)t_{i,j}(u-1)\cdots t_{i,j}(u-p+1), t_{k,l}(v)]=0.
\]
Then this follows immediately from \cite[Lemma 6.8]{BT18}.
\end{proof}

\begin{Remark}\label{Remark:b sij}
In the special case $\mu=(1,\dots,1)$,
we recall that the presentation is just Drinfeld presentation.
We define $b_i(u)$ via
\begin{align}\label{biu}
b_i(u)=\sum\limits_{r\geq 0}b_i^{(r)}u^{-r}:=d_i(u)d_i(u-1)\cdots d_i(u-p+1)    
\end{align}
for $i=1,\dots,n$ in terms of the diagonal Drinfeld generators, 
and it is shown that all elements $b_i^{(r)}$ belong to $Z(Y_n)$ (\cite[Lemma 5.7]{BT18}).
On the other hand, when $\mu=(n)$, we have $m=1$,
so that $D_{a;i,j}(u)=D_{1;i,j}(u)=t_{i,j}(u)$ and 
\begin{align}\label{siju}
B_{a;i,j}(u)=t_{i,j}(u)t_{i,j}(u-1)\cdots t_{i,j}(u-p+1)=:\sum\limits_{r\geq 0}s_{i,j}^{(r)}u^{-r}=s_{i,j}(u).  
\end{align}
Also, all of the coefficients $s_{i,j}^{(r)}$ belong to $Z(Y_n)$ (\cite[Lemma 6.8]{BT18}).
\end{Remark}

Here we give another proof for Lemma \ref{Lemma:Baij in center}
\begin{proof}[Proof of Lemma \ref{Lemma:Baij in center}]
Using Lemma \ref{lemma:DEFab-DEFa}, 
we reduce to proving that all coefficients of $B_{a;1,1}(u)$ and of $B_{a;1,2}(u)$ are central.
Let $\bar{\mu}:=(\mu_a,\mu_{a+1},\dots,\mu_m)$ and we put $k:=p_a(\mu)=\mu_1+\cdots+\mu_{a-1}$.
We have by \eqref{psik D} that
\[
{^\mu}B_{a;1,1}(u)=\psi_{k}({^{\bar{\mu}}}B_{1;1,1}(u)),~{^\mu}B_{a;1,2}(u)=\psi_{k}({^{\bar{\mu}}}B_{1;1,2}(u)).
\]
Hence the definition \eqref{formula:baij} and \eqref{quasiD} in conjunction with \cite[(6.32),(6.33)]{BT18} yield
\begin{align*}
{^{\bar{\mu}}}B_{1;1,1}(u)=&t_{1,1}(u)t_{1,1}(u-1)\cdots t_{1,1}(u-p+1)\\
=&d_1(u)d_1(u-1)\cdots d_{1}(u-p+1),\\
{^{\bar{\mu}}}B_{1;1,2}(u)=&t_{1,2}(u)t_{1,2}(u-1)\cdots t_{1,2}(u-p+1)\\
=&d_1(u)d_1(u-1)\cdots d_1(u-p+1)(e_1(u))^p.
\end{align*}
By applying \eqref{psik d} and \eqref{psik e}, 
we obtain
\[
{^\mu}B_{a;1,1}(u)=\psi_{k}({^{\bar{\mu}}}B_{1;1,1}(u))=d_{k+1}(u)d_{k+1}(u-1)\cdots d_{k+1}(u-p+1),\]
\[
{^\mu}B_{a;1,2}(u)=\psi_{k}({^{\bar{\mu}}}B_{1;1,2}(u))=d_{k+1}(u)d_{k+1}(u-1)\cdots d_{k+1}(u-p+1)(e_{k+1}(u))^p.
\]
Our result now follows from Lemma \ref{Lemma:EupFDE=0} together with Remark \ref{Remark:b sij}.
\end{proof}

Next we introduce the diagonal subalgebras
\begin{align}\label{Yaij0}
Y_{a;i,j}^0:=\kk[D_{a;i,j}^{(r)};~r>0]
\end{align}
of $Y_n$. 
By definition (see subsection \ref{subsection:shifted yangian}), 
we note that $Y_{a;i,j}^0$ is also subalgebra of $Y_n(\sigma)$.

\begin{Theorem}
Assume that $n\geq 2$.
For $1\leq a\leq m$ and all admissible $i,j$,
the algebra $Z(Y_n)\cap Y_{a;i,j}^0$ is an infinite rank polynomial algebra freely generated by the central elements
$\{B_{a;i,j}^{(rp)};~r>0\}$. This statement also describes the algebras $Z(Y_n(\sigma))\cap Y_{a;i,j}^0$.
We have that $B_{a;i,j}^{(rp)}\in {\rm F}_{rp-p} Y_n$ and 
\begin{align}\label{gr baijrp}
\gr_{rp-p} B_{a;i,j}^{(rp)}=(e_{a,a;i,j}t^{r-1})^p-\delta_{i,j}e_{a,a;i,j}t^{rp-p}. 
\end{align}
For $0<r<p$, we have that $B_{a;i,j}^{(r)}=0$.
For $r>p$ and $p\nmid r$,
we have that $B_{a;i,j}^{(r)}\in{\rm F}_{r-p-1} Y_n$ and it may be expressed as a polynomial in the elements $\{B_{a;i,j}^{(sp)};~0<s\leq \lfloor r/p\rfloor\}$.
\end{Theorem}
\begin{proof}
Let $\fg_{a;i,j}^0$ be the abelian subalgebra of $\fg$ spanned by $\{e_{a,a;i,j}t^r;~r\geq 0\}$.
By Theorem \ref{theorem:center of ugsigma},
one sees that
\begin{align}
 Z(\fg)\cap U(\fg_{a;i,j}^0)=\kk[(e_{a,a;i,j}t^{r})^p-\delta_{i,j}e_{a,a;i,j}t^{rp};~r\geq 0].   
\end{align}
We have that $\gr(Z(Y_n)\cap Y_{a;i,j}^0)\subseteq Z(\fg)\cap U(\fg_{a;i,j}^0)$.
On the other hand, we know that $B_{a;i,j}^{(rp+p)}$ belongs to $Z(Y_n)\cap Y_{a;i,j}^0$ by Lemma \ref{Lemma:Baij in center}.
Moreover, applying \cite[Lemma 2.9]{BT18} if $i=j$ or \cite[Lemma 2.11]{BT18} if $i\neq j$,
we see that $B_{a;i,j}^{(rp+p)}\in {\rm F}_{rp} Y_{a;i,j}^0$ and $\gr_{rp}B_{a;i,j}^{(rp+p)}=(e_{a,a;i,j}t^{r})^p-\delta_{i,j}e_{a,a;i,j}t^{rp}$.
Consequently, $\gr(Z(Y_n)\cap Y_{a;i,j}^0)=Z(\fg)\cap U(\fg_{a;i,j}^0)$ and the elements $\{B_{a;i,j}^{(rp)};~r>0\}$ are algebraically independent generators. 
These lemmas also show that $B_{a;i,j}^{(r)}=0$ for $0<r<p$ and that $B_{a;i,j}^{(r)}\in {\rm F}_{r-p-1}Y_{a;i,j}^0$ when $r\geq p$ with $p\nmid r$. 
Since it is central by Lemma \ref{Lemma:Baij in center} again,
it must be a polynomial in the elements $\{B_{a;i,j}^{(sp)};~0<s\leq \lfloor r/p\rfloor\}$.
The same arguments works for $Y_n(\sigma)$.
\end{proof}

\subsection{The center of $Z(Y_n(\sigma))$}
We will use the notation ${^\mu}Y_n(\sigma)$ to emphasize the choice of the admissible $\mu$.
We have already defined the Harish-Chandra center $Z_{\HC}({^\mu}Y_n(\sigma))$ in Section \ref{section name HC center}.
Also define the {\em $p$-center} $Z_p({^\mu}Y_n(\sigma))$ of ${^\mu}Y_n(\sigma)$ to be the subalgebra generated by
\begin{align*}
I_p&:=\{B_{a;i,j}^{(rp)};~1\leq a\leq m, 1\leq i,j\leq \mu_a, r>0\},\\
J_p&:=\{({^\sigma}E_{a,b;i,j}^{(r)})^p;~1\leq a<b\leq m, 1\leq i\leq \mu_a, 1\leq j\leq \mu_b, r>s_{a,b}^{\mu}\},\\
K_p&:=\{({^\sigma}F_{b,a;k,l}^{(r)})^p;~1\leq a<b\leq m, 1\leq l\leq \mu_a, 1\leq k\leq \mu_b, r>s_{b,a}^{\mu}\}.
\end{align*}
We have shown that $Z_p({^\mu}Y_n(\sigma))$ is a subalgebra of $Z({^\mu}Y_n(\sigma))$,
see Corollary \ref{corollary:central EF eleemnts ynsigma}
and Lemma \ref{Lemma:Baij in center}.
Note also by \eqref{gr of sigma EF abijrp} and \eqref{gr baijrp} that $\gr Z_p({^\mu}Y_n(\sigma))$
may be identified with the $p$-center $Z_p(\fg_\sigma)$ of $U(\fg_\sigma)$ from \eqref{generator of p-centerof ug sigma}.

We also need one more family of elements.
We let
\begin{align}\label{bcu}
bc(u):=\sum\limits_{r\geq 0}bc^{(r)}u^{-r}:=c(u)c(u-1)\cdots c(u-p+1).
\end{align}

\begin{Lemma}\label{Lemma:bcr in Zpyn sigma}
For any $r>0$, we have that $bc^{(r)}\in Z_p({^\mu}Y_n(\sigma))$.   
\end{Lemma}
\begin{proof}
Recalling \eqref{biu}, we claim that
\[(\ast)\ \ \ \ b_k^{(r)}~\text{belongs to the algebra generated by}~I_p\]
for all $1\leq k\leq n$ and $r>0$.
The result follows from the claim since we know from \eqref{power series: HC center drinfeld} 
that each $bc^{(r)}$ can be expressed as a polynomial in the elements $\{b_k^{(r)};~1\leq k\leq n, r>0\}$.

To prove the claim, we first consider the natural embedding $\eta: Y_{\mu_1}\hookrightarrow Y_n;~t_{i,j}^{(r)}\mapsto D_{1,i,j}^{(r)}=t_{i,j}^{(r)}$.
Clearly, $\eta(s_{i,j}(u))=B_{1;i,j}(u)$ and $\eta(b_k(u))=b_k(u)$ for $1\leq i,j\leq \mu_1$ and $1\leq k\leq \mu_1$.
Thanks to \cite[Theorem 6.9]{BT18}, we know $Z_p({^{(1^{\mu_1})}}Y_{\mu_1})=Z_p({^{(\mu_1)}}Y_{\mu_1})$, 
that is, $b_k^{(r)}$ belongs to the algebra generated by the elements $\{s_{i,j}^{(rp)};~1\leq i,j\leq \mu_1, r>0\}$.
Consequently, $(\ast)$ holds for $1\leq k\leq \mu_1$. 
Assume that $p_a(\mu)+1\leq k\leq p_{a+1}(\mu)$ for some $a\geq 2$.
Let $\bar{\mu}:=(\mu_a,\mu_{a+1},\dots,\mu_m)$. 
We apply Lemma \ref{Lemma:property of psi} to see that
$\psi_{p_a(\mu)}(b_l(u))=b_{p_a(\mu)+l}(u)$ for all $1\leq l\leq \mu_a$
and $\psi_{p_a(\mu)}({^{\bar{\mu}}}B_{1;i,j}(u))={^\mu}B_{a;i,j}(u)$.
The foregoing observation implies that $b_l^{(r)}$ belongs to the algebra generated by the elements $\{{^{\bar{\mu}}}B_{1;i,j}^{(rp)};~1\leq i,j\leq \mu_a, r>0\}$.
As a result, we obtain $(\ast)$ for $p_a(\mu)+1\leq k\leq p_{a+1}(\mu)$.
\end{proof}

From the definition \eqref{bcu}, 
it follows that each $bc^{(r)}$ can be expressed as a polynomial in the elements $\{c^{(s)};~s>0\}$,
so that it belongs to $Z_{\HC}({^\mu}Y_n(\sigma))$.
By Lemma \ref{Lemma:bcr in Zpyn sigma}, we have just shown that $bc^{(r)}\in Z_{\HC}({^\mu}Y_n(\sigma))\cap Z_p({^\mu}Y_n(\sigma))$.
Moreover, \cite[Lemma 5.10]{BT18} shows that $bc^{(rp)}\in {\rm F}_{rp-p}{^\mu}Y_n(\sigma)$ and 
\begin{align}\label{gr rp bc}
\gr_{rp-p}bc^{(rp)}=z_{r-1}^p-z_{rp-p}.   
\end{align}

\begin{Theorem}\label{Theorem:main theorem}
The center $Z({^\mu}Y_n(\sigma))$ is generated by $Z_{\HC}({^\mu}Y_n(\sigma))$ and $Z_p({^\mu}Y_n(\sigma))$. 
Moreover,
\begin{itemize}
    \item[(1)] $Z_{\HC}({^\mu}Y_n(\sigma))$ is the free polynomial algebra generated by $\{c^{(r)};~r>0\}$;
    \item[(2)] $Z_p({^\mu}Y_n(\sigma))$ is the free polynomial algebra generated by $I_p\cup J_p\cup K_p$;
    \item[(3)] $Z({^\mu}Y_n(\sigma))$ is the free polynomial algebra generated by \[
    J_p\cup K_p\cup\{c^{(r)}, B_{a;i,j}^{(rp)};~1\leq a\leq m, 1\leq i,j\leq \mu_a, (a,i,j)\neq (1,1,1), r>0\};
    \]
    \item[(4)] $Z_{\HC}({^\mu}Y_n(\sigma))\cap Z_p({^\mu}Y_n(\sigma))$ is the free polynomial algebra generated by $\{bc^{(rp)};~r>0\}$.
\end{itemize}
\end{Theorem}
\begin{proof}
(1) This is Proposition \ref{prop-HC center gb parabolic generators},
see also \cite[Theorem 5.11(1)]{BT18}.

(2) We just need to check that they are algebraically independent.
This follows from \eqref{gr of sigma EF abijrp} and \eqref{gr baijrp},
because they are lifts of the algebraically independent generators of the $p$-center $Z_p(\fg_\sigma)$ of the associated graded algebra $U(\fg_\sigma)$ from \eqref{generator of p-centerof ug sigma}.

(3) Let $Z$ be the subalgebra of $Z({^\mu}Y_n(\sigma))$ generated by the given elements. We have that
\[
\gr Z\subseteq \gr Z({^\mu}Y_n(\sigma))\subseteq Z(\gr{^\mu}Y_n(\sigma))=Z(\fg_\sigma)
\]
Using again \eqref{gr of sigma EF abijrp} and \eqref{gr baijrp} in conjunction with Theorem \ref{theorem:center of ugsigma},
we know that generators of $Z$ are lifts of the algebraically independent generators of $Z(\fg_\sigma)$.
Hence, they are algebraically independent and $Z=Z({^\mu}Y_n(\sigma))$.
Moreover, (1) and (2) imply in particular that the center $Z({^\mu}Y_n(\sigma))$ is generated by $Z_{\HC}({^\mu}Y_n(\sigma))$ and $Z_p({^\mu}Y_n(\sigma))$.

(4) The algebraic independence follows
immediately from \eqref{gr rp bc}.
Observe that all $bc^{(r)}$ belong to $Z_p({^\mu}Y_n(\sigma))$ by Lemma \ref{Lemma:bcr in Zpyn sigma}.
We first show that $Z_p({^\mu}Y_n(\sigma))$ can be freely generated by the elements
\[
(\ast\ast)\ \ \ J_p\cup K_p\cup \{bc^{(rp)}, B_{a;i,j}^{(rp)};~1\leq a\leq m, 1\leq i,j\leq \mu_a, (a,i,j)\neq (1,1,1), r>0\}.
\]
We use \eqref{gr of sigma EF abijrp}, \eqref{gr baijrp} and \eqref{gr rp bc} to pass to the associated graded algebra,
it remains to show that $Z_p(\fg_\sigma)$ is freely generated by the elements
\begin{align*}
\{(e_{a,a;i,j}t^{r})^p-\delta_{i,j}e_{a,a;i,j}t^{rp};~1\leq a\leq m,1\leq i,j\leq \mu_a,&(a,a,i,j)\neq (1,1,1,1), r\geq 0\}\\
\cup\{(e_{a,b;i,j}t^{r})^p;~1\leq a<b\leq m, 1\leq i\leq \mu_a,& 1\leq j\leq \mu_b, r\geq s_{a,b}^{\mu}\}\\
 \cup \{(e_{b,a;i,j}t^{r})^p;~1\leq a<b\leq m,&1\leq j\leq \mu_a, 1\leq i\leq \mu_b, r\geq s_{b,a}^{\mu}\}\\
 &\cup\{z_r^p-z_{rp};~r\geq 0\}.
\end{align*}
This is easily seen by comparing them to the algebraically independent generators from \eqref{generator of p-centerof ug sigma}.
We know from (3) that all the elements in $(\ast\ast)$ different from 
$bc^{(rp)}$ are algebraically independent of anything in $Z_{\HC}({^\mu}Y_n(\sigma))$.
Our assertion thus follows from the fact that all $bc^{(r)}$ belong to $Z_p({^\mu}Y_n(\sigma))\cap Z_{\HC}({^\mu}Y_n(\sigma))$.
\end{proof}

Recall that $\mu=(\mu_1,\dots,\mu_m)$. 
Suppose that $\mu_b>1$ for some $1\leq b\leq m$ and we decompose $\mu_b=x+y$ for some positive integers $x,y$, 
and define a finer composition $\nu$ of length $m+1$ by setting $\nu_i=\mu_i$ for all $1\leq i\leq b-1$, $\nu_b=x$,
$\nu_{b+1}=y$, $\nu_{j+1}=\mu_j$ for all $b+1\leq j\leq m$; that is,
\begin{align}\label{def of nu}
\nu=(\mu_1\,\dots,\mu_{b-1},x,y,\mu_{b+1},\dots,\mu_m),
\end{align}
which is also admissible to $\sigma$ by definition.
Consider the Gauss decomposition of the matrix $T(u)$ with respect to $\mu$ and $\nu$, respectively:
\[
T(u)={^\mu}F(u){^\mu}D(u){^\mu}E(u)={^\nu}F(u){^\nu}D(u){^\nu}E(u),
\]
where the matrices are block matrices as described in Section \ref{subsection name:Gauss decomp}.
Denote by ${^\mu}D_a$ and ${^\nu}D_a$ the $a$-th diagonal matrices in ${^\mu}D(u)$ and ${^\nu}D(u)$, respectively.
Similarly, let ${^\mu}E_a$ and ${^\mu}F_a$ denote the matrices in the $a$-th upper and the $a$-th lower diagonal of $^{\mu}E(u)$
and ${^\mu}F(u)$, respectively;
${^\nu}E_a$ and ${^\nu}F_a$ denote the matrices in the $a$-th upper and the $a$-th lower diagonal of ${^\nu}E(u)$
and ${^\nu}F(u)$, respectively.

We may apply Lemma \ref{lemma:grYnsigma=ugsigma} and the following result to see that the algebra ${^\mu}Y_n(\sigma)$
is independent of the choice of the admissible shape $\mu$ (see \cite[page 150]{BK06}).
\begin{Lemma}\cite[Lemma 3.1]{BK06}\label{Lemma: split block lemma}
In the above notation, define an $x\times x$-matrix $A$, an $x\times y$-matrix $B$, 
a $y\times x$-matrix $C$ and a $y\times y$-matrix $D$ from the equation  
\[
{^\mu D_b}=\left(\begin{array}{cc}
     I_x&0  \\
     C& I_y
\end{array}\right)\left(\begin{array}{cc}
     A&0  \\
     0& D
\end{array}\right)\left(\begin{array}{cc}
     I_x&B  \\
     0& I_y
\end{array}\right).
\]
Then
\begin{itemize}
    \item[(i)] ${^\nu}D_a={^\mu}D_a$ for $a<b$, ${^\nu}D_b=A$, ${^\nu}D_{b+1}=D$, and ${^\nu}D_c={^\mu}D_{c-1}$ for $c>b+1$;
    \item[(ii)] ${^\nu}E_a={^\mu}E_a$ for $a<b-1$, ${^\nu}E_{b-1}$ is the submatrix consisting of the first $x$ columns of ${^\mu}E_{b-1}$,
    ${^\nu}E_b=B$, ${^\nu}E_{b+1}$ is the submatrix consisting of the last $y$ rows of ${^\mu}E_b$, and ${^\nu}E_c={^\mu}E_{c-1}$ for $c>b+1$;
    \item[(iii)] ${^\nu}F_a={^\mu}F_a$ for $a<b-1$, ${^\nu}F_{b-1}$ is the submatrix consisting of the first $x$ rows of ${^\mu}F_{b-1}$,
    ${^\nu}F_b=C$, ${^\nu}F_{b+1}$ is the submatrix consisting of the last $y$ columns of ${^\mu}F_b$, and ${^\nu}F_c={^\mu}F_{c-1}$ for $c>b+1$.
\end{itemize}
\end{Lemma}

\begin{Remark}\label{Remark:highly plausible}
In the special case $\mu=(1^n)$, the definition of the $p$-center was given by Brundan and Topley in \cite[(5.18)]{BT18}.
In the following, we show that the $p$-center $Z_p({^\mu}Y_n(\sigma))$ of $Y_n(\sigma)$ is independent of the choice of the admissible shape $\mu$.
Hence we obtain a generalization of \cite[Theorem 6.9]{BT18}.
\end{Remark}

\begin{Lemma}\label{lemma: p-center stable under permutation auto}
Let $w$ be any permutation automorphism of $Y_n$ from Section \ref{section: auto}.
Then we have
\[
w(Z_p({^{(1^n)}}Y_n))=w(Z_p({^{(n)}}Y_n))=Z_p({^{(n)}}Y_n)=Z_p({^{(1^n)}}Y_n).
\]
\end{Lemma}
\begin{proof}
Thanks to \cite[Theorem 6.9]{BT18}, we have
$Z_p({^{(1^n)}}Y_n)=Z_p({^{(n)}}Y_n)$.
It actually suffices to check just the second equality.
As $Z_p({^{(n)}}Y_n)$ is generated by the elements $\{s_{i,j}^{(rp)};~1\leq i,j\leq n, r>0\}$,
our assertion follows from the definition that $w(s_{i,j}(u))=s_{w(i),w(j)}(u)$.
\end{proof}

\begin{Lemma}\label{Lemma:psi pcenter to p-center}
Let $\bar{\mu}:=(\mu_a,\mu_{a+1},\dots,\mu_m)$ and we put $k:=p_a(\mu)=\mu_1+\cdots+\mu_{a-1}$.
Then we have 
\[
\psi_k(Z_p({^{\bar{\mu}}}Y_{n-k}))\subseteq Z_p({^{\mu}}Y_n).
\]
\end{Lemma}
\begin{proof}
Notice that the $p$-center $Z_p({^{\bar{\mu}}}Y_{n-k})$ is generated by the elements
\begin{align*}
\{B_{c;i,j}^{(rp)};~1\leq c\leq m-a&+1, 1\leq i,j\leq\mu_{a+c-1}, r>0\}\\
\cup\{(E_{c,d;i,j}^{(r)})^p;~&1\leq c<d\leq m-a+1, 1\leq i\leq\mu_{a+c-1}, 1\leq j\leq\mu_{a+d-1}, r>0\}\\
\cup \{&(F_{d,c;j,i}^{(r)})^p;~1\leq c<d\leq m-a+1, 1\leq i\leq\mu_{a+c-1}, 1\leq j\leq\mu_{a+d-1}, r>0\}.
\end{align*}
We apply Lemma \ref{Lemma:property of psi} in conjunction with \eqref{EFabij r induction formula} to see that
\[
\psi_k(B_{c;i,j}^{(rp)})=B_{c+k;i,j}^{(rp)}, \psi_k((E_{c,d;i,j}^{(r)})^p)=(E_{c+k,d+k;i,j}^{(r)})^p ~\text{and}~ \psi_k((F_{d,c;j,i}^{(r)})^p)=(F_{d+k,c+k;j,i}^{(r)})^p.
\]
Since all the right hand side elements of the above equalities are contained in the $p$-center $Z_p({^{\mu}}Y_n)$,
the assertion follows.
\end{proof}

\begin{Proposition}\label{prop: p-center of Yangian all equal}
For any composition $\mu$ of $n$,
we have that $Z_p({^\mu}Y_n)=Z_p({^{(1^n)}}Y_n)$. 
\end{Proposition}
\begin{proof}
Since $\gr Z_p({^\mu}Y_n)$ can be identified with the $p$-center $Z_p(\fg)$ for any $\mu$,
it suffices to show that $Z_p({^\mu}Y_n)\subseteq Z_p({^{(1^n)}}Y_n)$.
We prove by applying induction on the length of the shape $\mu$.
If $\mu_j=1$ for all $j$, then we are done.
Otherwise, suppose that $\mu_b>1$ for some $1\leq b\leq m$.
Then we split the $b$-th block to get a new shape $\nu$ just like in \eqref{def of nu}.

Two cases arise:

(a) For all $a=1,\dots,m$, $r>0$ and admissible $i,j$,
the element $B_{a;i,j}^{(rp)}$ belongs to $Z_p({^{(1^n)}}Y_n)$. 

We put $k:=p_a(\mu)=\mu_1+\cdots+\mu_{a-1}$.
Using \eqref{psik D} and observing $D_{1;i,j}(u)=t_{i,j}(u)$, we obtain
\[
B_{a;i,j}^{(rp)}=\psi_k(B_{1;i,j}^{(rp)})=\psi_k(s_{i,j}^{(rp)}).
\]
Notice that $s_{i,j}^{(rp)}\in Z_p({^{(n-k)}}Y_{n-k})=Z_p({^{(1^{n-k})}}Y_{n-k})$.
Then Lemma \ref{Lemma:psi pcenter to p-center} yields the assertion.

(b) For all $1\leq c<d\leq m$, $r>0$ and admissible $i,j$,
the elements $(E_{c,d;i,j}^{(r)})^p$ and $(F_{d,c;j,i}^{(r)})^p$ belong to $Z_p({^{(1^n)}}Y_n)$.

Using Lemma \ref{lemma: p-center stable under permutation auto} and Lemma \ref{lemma:DEFab-DEFa},
the proof of this reduces to checking that the element $(E_{c;i,1}^{(r)})^p$ belongs to $Z_p({^{(1^n)}}Y_n)$ for all $1\leq c\leq m-1$ and admissible $i$.
By the induction hypothesis, we only need to check $(E_{c;i,1}^{(r)})^p\in Z_p({^{\nu}}Y_n)$.
Lemma \ref{Lemma: split block lemma} yields $(E_{c;i,1}^{(r)})^p\in Z_p({^{\nu}}Y_n)$ for every $c\neq b$.
It remains to show that $(E_{b;i,1}^{(r)})^p\in Z_p({^{\nu}}Y_n)$.
Again by Lemma \ref{Lemma:psi pcenter to p-center}, we may assume that $b=1$.
We have by \eqref{quasiE} that $E_1(u)=D_1'(u)T_{1,2}(u)$.
Direct computation shows that the permutation automorphism by $i\mapsto \mu_1$ maps $E_{1;i,1}(u)\mapsto E_{1;\mu_1,1}(u)$.
Observing that $(E_{1;\mu_1,1}^{(r)})^p\in Z_p({^{\nu}}Y_n)$ by Lemma \ref{Lemma: split block lemma}(ii) in conjunction with Lemma \ref{lemma: p-center stable under permutation auto}, we thus obtain $(E_{b;i,1}^{(r)})^p\in Z_p({^{\nu}}Y_n)$. 
One argues similarly for $(F_{d,c;j,i}^{(r)})^p$.
\end{proof}

\begin{Lemma}\label{lemma:iota send center = center}
Let $\dot{\sigma}=(\dot{s}_{i,j})_{1\leq i,j\leq n}$
be another shift matrix satisfying $s_{i,i+1}+s_{i+1,i}=\dot{s}_{i,i+1}+\dot{s}_{i+1,i}$ for all $i=1,\dots,n-1$,
and let $\iota$ be the corresponding change of shift matrix isomorphism from Section \ref{section: auto}.
Then we have
\[
\iota(Z_p({^\mu}Y_n(\sigma)))=Z_p({^\mu}Y_n(\dot{\sigma})).
\]
\end{Lemma}
\begin{proof}
Using \eqref{sigma Eabij} and \eqref{sigma Fbaji} and induction on $b-a$,
it is easy to see that
\[
\iota({^\sigma}E_{a,b;i,j}^{(r)})={^{\dot{\sigma}}}E_{a,b;i,j}^{(r-s_{a,b}^{\mu}+\dot{s}_{a,b}^{\mu})}, \iota({^\sigma}F_{b,a;j,i}^{(r)})={^{\dot{\sigma}}}F_{b,a;j,i}^{(r-s_{b,a}^{\mu}+\dot{s}_{b,a}^{\mu})}.
\]
Then the assertion follows directly from the definition of the $p$-center.
\end{proof}

\begin{Theorem}\label{theorem: p-center of shifted Yangian all equal}
$Z_p({^\mu}Y_n(\sigma))$ is independent of the choice of the admissible shape $\mu$.
\end{Theorem}
\begin{proof}
We have seen already that $\gr Z_p({^\mu}Y_n(\sigma))$
may be identified with the $p$-center $Z_p(\fg_\sigma)$,
which is independent of the choice of the admissible shape $\mu$,
so that it is enough to show that $Z_p({^{(1^n)}}Y_n(\sigma))\subseteq Z_p({^\mu}Y_n(\sigma))$.
We prove the claim by induction on the length of $\mu$.
Similar to the proof of Proposition \ref{prop: p-center of Yangian all equal},
we can get a new admissible shape $\nu$ by splitting the $b$-th block of $\mu$.
By the induction hypothesis,
we need to check that $Z_p({^\nu}Y_n(\sigma))\subseteq Z_p({^\mu}Y_n(\sigma))$.

Let ${^\mu}Y_n(\sigma)^0$ denote the subalgebra of ${^\mu}Y_n(\sigma)$ generated by the set $I(\sigma)$ and let ${^\mu}Y_n(\sigma)^{\geq 0}$ denotes the parabolic subalgebra of ${^\mu}Y_n(\sigma)$ generated by $I(\sigma)$ and $J(\sigma)$.
By Lemma \ref{Lemma: split block lemma} and Proposition \ref{prop: p-center of Yangian all equal} , we have 
\[
Z_p({^\nu}Y_n(\sigma))\ni B_{a;i,j}^{(rp)}\in Z_p({^\nu}Y_n)\cap {^\mu}Y_n(\sigma)^0=Z_p({^\mu}Y_n)\cap {^\mu}Y_n(\sigma)^0\subseteq  Z_p({^\mu}Y_n(\sigma)).
\]
This implies that ${^\nu}I_p\subseteq Z_p({^\mu}Y_n(\sigma))$.
To prove that ${^\nu}J_p\subseteq Z_p({^\mu}Y_n(\sigma))$,
we first assume that $\sigma$ is lower triangular.
In this case,
we apply again Lemma \ref{Lemma: split block lemma} and Proposition \ref{prop: p-center of Yangian all equal} to see that
\[
Z_p({^\nu}Y_n(\sigma))\ni ({^\sigma}E_{a,b;i,j}^{(r)})^p=(E_{a,b;i,j}^{(r)})^p\in Z_p({^\nu}Y_n)\cap {^\mu}Y_n(\sigma)^{\geq 0}=Z_p({^\mu}Y_n)\cap {^\mu}Y_n(\sigma)^{\geq 0}\subseteq Z_p({^\mu}Y_n(\sigma)).
\]
For general $\sigma$, this follows using the change of shift matrix isomorphism from Lemma \ref{lemma:iota send center = center}.
One argues similarly to obtain ${^\nu}K_p\subseteq Z_p({^\mu}Y_n(\sigma))$.
\end{proof}

\bigskip
\noindent
\textbf{Acknowledgment.} 
We would like to thank Yung-Ning Peng for helpful discussions.
We are particularly grateful to the referee for a very detailed reading of this paper and many constructive comments and suggestions which improve the original manuscript.
This work is supported by the National Natural
Science Foundation of China (Grant Nos. 11801394, 12461005), the Natural Science Foundation of Hubei Province (No. 2025AFB716) and
the Fundamental Research Funds for the Central Universities (Nos. CCNU24JC001, CCNU25JC025, CCNU25JCPT031).


\bigskip
\begin{thebibliography}{}
\bibitem[BK1]{BK05} J. Brundan and A. Kleshchev, {\em Parabolic presentations of the Yangian $Y(\gl_{n})$}.
Comm. Math. Phys. \textbf{254} (2005), 191--220.
\bibitem[BK2]{BK06} J. Brundan and A. Kleshchev, {\em Shifted Yangians and finite $W$-algebras}.
Adv. Math. \textbf{200} (2006), 136--195.
\bibitem[BK3]{BK08} J. Brundan and A. Kleshchev, {\em Representations of shifted Yangians and finite $W$-algebras}.
Mem. Amer. Math. Soc. \textbf{196} (2008), no. 918.
\bibitem[BT]{BT18} J. Brundan and L. Topley, {\em The $p$-centre of Yangians and shifted Yangians}.
Mosc. Math. J. \textbf{18} (2018), no. 4, 617-657.
\bibitem[CH1]{CH23-1} H. Chang and H. Hu, {\em A note on the center of the super Yangian $Y_{M|N}(\mathfrak{s})$}.
J. Algebra \textbf{633} (2023), 648--665.
\bibitem[CH2]{CH23-2} H. Chang and H. Hu, {\em The centre of the modular super Yangian $Y_{m|n}$}.
J. London Math. Soc. \textbf{107} (2023), 1074--1109.
\bibitem[D1]{D85} V. G. Drinfeld, {\em Hopf algebras and the quantum Yang--Baxter equation}.
Soviet Math. Dokl. \textbf{32} (1985), 254--258.
\bibitem[D2]{D88} V. G. Drinfeld, {\em A new realization of Yangians and quantized affine algebras}.
Soviet Math. Dokl. \textbf{36} (1988), 212--216.
\bibitem[GR]{GR97} I. Gelfand and V. Retakh, {\em Quasideterminants. I}.
Selecta Math. \textbf{3} (1997), 517--546.
\bibitem[GT1]{GT2019} S. M. Goodwin and L. Topley, {\em Modular finite W-algebras}.
Int. Math. Res. Not. IMRN \textbf{19} (2019), 5811--5853.
\bibitem[GT2]{GT19} S. M. Goodwin and L. Topley, {\em Minimal-dimensional representations of reduced enveloping algebras for $\mathfrak{gl}_n$}.
Compos. Mathematica \textbf{155} (2019), 1594--1617.
\bibitem[GT3]{GT21} S. M. Goodwin and L. Topley, {\em Restricted shifted Yangians and restricted finite $W$-algebras}.
Trans. Amer. Math. Soc., Ser. B \textbf{8} (2021), 190--228.
\bibitem[Jan]{Jan97} J. C. Jantzen, {\em Representations of Lie algebras in prime characteristic}.
in: A. Broer (Ed.), Representation Theories and Algebraic Geometry, in: Proceedings, Montreal, NATO ASI Series, vol. C 514, Kluwer, Dordrecht, 1998, pp. 185--235.
\bibitem[MNO]{MNO96} A. Molev, M. Nazarov and G. Olshanskii, {\em Yangians and classical Lie algebras}.
Russian Math. Surveys \textbf{51} (1996), 205--282.
\bibitem[Mol]{Mol07} A. Molev, {\em Yangians and Classical Lie Algebras}.
Mathematical Surveys and Monographs, vol. \textbf{143}, American Mathematical Society, Providence, RI, 2007.
\bibitem[Peng1]{Peng11} Y. Peng, {\em Parabolic presentations of the super Yangian $Y(\gl_{M|N})$}.
Commun. Math. Phys. \textbf{307} (2011), 229--259.
\bibitem[Peng2]{Peng21} Y. Peng, {\em Finite $W$-superalgebras via super Yangians}.
Adv. Math. \textbf{377} (2021), 107459.
\bibitem[Pr]{Pre02} A. Premet, {\em Special transverse slices and their enveloping algebras}.
Adv. Math. \textbf{170} (2002), 1--55.
\bibitem[TF]{TF79} L. Takhtadzhyan and L. Faddeev, {\em The quantum method of the inverse problem and the heisenberg XYZ-model}.
Russ. Math. Serv. \textbf{34} (1979), 11--68.
\end{thebibliography}
\end{document}